\documentclass[11pt,reqno]{amsart}
\usepackage{geometry}                
\usepackage{graphicx}
\usepackage{amssymb,amsthm,amsmath}
\usepackage{epstopdf}
\usepackage{enumerate}  
\usepackage[low-sup]{subdepth}
\usepackage{todonotes}
\setuptodonotes{fancyline,color=green!50}
\usepackage[shortlabels]{enumitem}

\usepackage{algorithm}
\usepackage{algpseudocode}
\algrenewcommand\algorithmicrequire{\textbf{Input:}}
\algrenewcommand\algorithmicensure{\textbf{Output:}}
\newcommand{\algorithmicbreak}{\textbf{break}}
\newcommand{\Break}{\State \algorithmicbreak}
 
\usepackage{bm}

\usepackage{tikz}
\usetikzlibrary{matrix}
\tikzset{
  delim options/.style={append after command={[#1]}},
  lr delims/.style={
    every delim around tikz cells/.append style={
      every left delimiter/.append style={#1},
      every right delimiter/.append style={#1}}},
  ab delims/.style={
    every delim around tikz cells/.append style={
      every above delimiter/.append style={#1},
      every below delimiter/.append style={#1}}},
  delim xshift/.style={
    every delim around tikz cells/.append style={
      every left delimiter/.append style={xshift={#1}},
      every right delimiter/.append style={xshift={-(#1)}}}},
  delim yshift/.style={
    every delim around tikz cells/.append style={
      every above delimiter/.append style={yshift={-(#1)}},
      every below delimiter/.append style={yshift={#1}}}},
  lr delim/.style args={#1#2 and #3#4 around #5 to #6}{
    append after command={
      {{[local bounding box=@] (\tikzlastnode-#5.north west)(\tikzlastnode-#6.south east)}
      (@)[every delim around tikz cells/.try,
          every left delimiter/.append style={#2},
          every right delimiter/.append style={#4},
          late options={left delimiter={#1},right delimiter={#3}}]}[]}},
  ab delim/.style args={#1#2 and #3#4 around #5 to #6}{
    append after command={
      {{[local bounding box=@] (\tikzlastnode-#5.north west)(\tikzlastnode-#6.south east)}
      (@)[every delim around tikz cells/.try,
          every above delimiter/.append style={#2},
          every below delimiter/.append style={#4},
          late options={above delimiter={#1},below delimiter={#3}}]}[]}},
}

\newcommand\mydots{\hbox to 1em{.\hss.\hss.}}

\renewcommand{\a}{\alpha}
\renewcommand{\b}{\beta}
\renewcommand{\c}{\gamma}

\newcommand{\kb}{\mathbf{k}}
\newcommand{\Kbb}{\mathbb{K}}

\newcommand{\Pbb}{\mathbb{P}}

\newcommand{\Rk}{\mathcal{R}}

\DeclareMathOperator{\Syz}{Syz}
\DeclareMathOperator{\lt}{lt}

\DeclareMathOperator{\lcm}{lcm}

\newcommand\red[1]{\textcolor{red}{\textbf{\textit{{#1}}}}}

\newtheorem*{Notation}{Notation}

\newtheorem*{Note}{Note}

\newtheorem{definition}{Definition}[section]

\newtheorem{proposition}{Proposition}[section]

\newtheorem{corollary}{Corollary}[section]

\newtheorem{theorem}{Theorem}[section]

\newtheorem{example}{Example}[section]

\newtheorem{lemma}{Lemma}[section]

\title[Free Resolution of Rees Algebra of Monomial Ideals]{On the minimal free resolution of the Rees algebra of tri-generated bivariate monomial ideals}
\author[R. Iglesias \and M. Orth \and E. S\'aenz-de-Cabez\'on \and W.M. Seiler]{Rodrigo Iglesias \and Matthias Orth \and Eduardo S\'aenz-de-Cabez\'on \and Werner M. Seiler}
\date{}                                           

\begin{document}
\maketitle
\begin{abstract}
Let $I$ be a monomial ideal in two variables generated by three monomials and let $\Rk(I)$ be its Rees ideal. We describe an algorithm to compute the minimal generating set of $\Rk(I)$. Based on the data obtained by this algorithm, we build a graph that encodes the minimal free resolution of $\Rk(I)$. We explicitly describe the modules and differentials on the minimal free resolution of $\Rk(I)$.
\end{abstract}

\section{Introduction}
Let $I\subseteq\Kbb[t_1,\dots,t_n]$ be an ideal in a polynomial ring on $n$ indeterminates. The Rees algebra of $I$ encodes important algebraic information about $I$ and geometrical information about the variety defined by $I$. In particular, the Rees algebra is used to compute integral closure of powers of ideals and the asymptotic behaviour of these powers and products of ideals \cite{BC17}; the homological properties of the Rees algebra characterise for instance ideals having linear powers, or linear products \cite{R01,BCV15,BC17}. In the geometric setting, the Rees algebra is a fundamental tool in the resolution of singularities of algebraic varieties. A recent application of Rees algebras is to the method of moving curves for the implicitization of a rational parametrization, which was developed by Sederberg and others in \cite{SC95,SGD97}. In \cite{C08},  Cox established a connection between this problem and the computation of the defining ideal of the Rees algebra of the ideal associated to the parametrization. After this connection was established, several cases have been analysed. In particular, in \cite{CortadellasDAndrea}, the authors study the case of monomial plane curves. They give a complete description of the minimal free resolution of the Rees algebra of the parametrization by means of an arithmetical procedure on the data defining the curve.

In general, little is known about the defining ideals and the structure of
free resolutions of Rees algebras, therefore the work in
\cite{CortadellasDAndrea} and others like \cite{V95,V08,GV09,FL15,LM06,N21}
are valuable contributions to the understanding of this important
object. This paper is a contribution to this area that extends some of the
previous work on Rees algebras. 
We develop an algorithm that
computes the defining ideal of the Rees algebra of monomial ideals in two
variables minimally generated by three monomials (tri-generated). The main
result of the first part of the paper is Algorithm~\ref{general:algorithm},
see Theorem~\ref{thm:general}. As a particular case, we apply our methods
to the problem of monomial (projective) plane curve parametrizations, as studied in
\cite{CortadellasDAndrea}. Using the intermediate computations of this algorithm, we describe a graph that encodes all the data in the minimal free resolution of the Rees algebra of the ideal.
This resolution is described in the main result of the paper, Theorem \ref{th:minimal_resolution}.
In the process of proving this result, we find a minimal Gr\"obner basis for the defining ideal of the Rees algebra of tri-generated monomial ideals in two variables for a suitable
term order.

The outline of the paper is the following: In Section
\ref{sec:preliminaries} we give the necessary preliminaries and basic
notations on Rees algebras. The main algorithm is
described in Section \ref{sec:mainAlgorithm}, and we apply it to the
particular case of parametrizations of monomial plane curves in Section
\ref{sec:monomialCurveParametrization}. Finally, Section
\ref{sec:resolution} is devoted to the description of the minimal free
resolution of the Rees algebra. This section starts with the description of the graph that encodes the minimal free resolution of the Rees algebra of the ideal (Section \ref{sec:graph}), then we obtain a Gr\"obner basis of the ideal (Section \ref{sec:grobner}), a non-minimal free resolution (Section \ref{sec:nm_resolution}) and finally the minimal free resolution (Section \ref{sec:minimal_resolution}). 

\section{Preliminaries and problem setting}\label{sec:preliminaries}
Let $I\subset R=\Kbb[T_1,\dots,T_n]=\Kbb[\mathbf{T}]$ be a monomial ideal, let $G(I)$ be the unique minimal set of monomial generators of $I$, and let $z$ be an extra variable. The \emph{Rees algebra} of $I$, denoted by $\Rk_I\subset R[z]$ is the $R$-subalgebra
\[
\Rk_I=R\oplus Iz\oplus I^2z^2\oplus\cdots
\]
We can describe the Rees algebra in the following way. Let $S$ denote the polynomial ring $S=\Kbb[T_1,\dots,T_n][X_w\mid w\in G(I)]=\Kbb[\mathbf{T}][\mathbf{X}]$. The {\em Rees map} is the $R$-algebra map $\phi:S\rightarrow R[z]$ given by $\phi(X_w)=wz$. The kernel of $\phi$ is called the {\em Rees ideal} of $I$ (or the {\em defining ideal of the Rees algebra of $I$)}, denoted by $\Rk(I)$. The Rees algebra of $I$ is given by the quotient $\Rk_I\simeq S/\Rk(I)$, which is isomorphic to the image of $\phi$. In addition to the Rees map, we will make use of the {\em toric map} $\psi:S\rightarrow R$ given by $\psi(X_w)=w$. The generators of $\Rk(I)=\ker(\phi)$ are binomials of the form $\mathbf{T}^{\a}\mathbf{X}^{\b}-\mathbf{T}^{\a'}\mathbf{X}^{\b'}$, where $\phi(\mathbf{T}^{\a}\mathbf{X}^{\b})=\phi(\mathbf{T}^{\a'}\mathbf{X}^{\b'})$. These correspond, by dropping the powers of $z$, to binomials such that $\psi(\mathbf{T}^{\a}\mathbf{X}^{\b})=\psi(\mathbf{T}^{\a'}\mathbf{X}^{\b'})$ where $\vert \b\vert=\vert \b'\vert$.

\begin{example}\label{ex:monomialParametrization}
In the case of a monomial plane curve, we start with a parametrization $\varphi:\Pbb^1_\Kbb\rightarrow \Pbb^2_\Kbb$ given by $(t_0:t_1)\mapsto(t_0^d:t_0^{u}t_1^{d-u}:t_1^d)$. We set now $R=\Kbb[T_0,T_1]$ and $S=\Kbb[T_0,T_1,X_0,X_1,X_2]$. The ideal of the parametrization is the monomial ideal $I\subseteq R$ given by $I=\langle T_0^d, T_0^{u}T_1^{d-u},T_1^d\rangle$ and the Rees map is given by $\phi(X_0)=T_0^dz,\,\phi(X_1)=T_0^{u}T_1^{d-u}z,\,\phi(X_2)=T_1^dz$. The Rees algebra of $I$ is fully studied in \cite{CortadellasDAndrea}.
\end{example}

The main approach for the study of $\Rk_I$ is to obtain a minimal generating
set of the Rees ideal $\Rk(I)$, which we denote $G\left(\Rk(I)\right)$. The binomials
in such a minimal generating set are called {\em essential} or {\em
  indispensable} binomials, and obtaining them is difficult in general, both for toric
and for Rees ideals, cf. \cite{CKT07,OV10,CortadellasDAndrea}. The minimal
generators of $\Rk(I)$ are given by relations among powers of the minimal
generators of $I$, these are essentially minimal first syzygies of $I^t$
for each $t$, but also some divisibility relations among powers of minimal
generators of $I$ might give rise to indispensable binomials.


Let $I\subseteq R=\Kbb[T_0,T_1]=\Kbb[\mathbf{T}]$ be a monomial ideal
minimally generated by the set $G(I)=\{
\mathbf{T}^{\a},\mathbf{T}^{\b},\mathbf{T}^{\c}\}$. Let
$h=T_0^{\mu_0}T_1^{\mu_1}$ be the greatest common divisor of the elements
of $G(I)$ and $I'=\langle \frac{g}{h}\vert g\in G(I)\rangle$. We have that
$s\cdot \mathbf{T}^i-s'\cdot \mathbf{T}^j=0\iff s\cdot
\frac{\mathbf{T}^i}{h}-s'\cdot \frac{\mathbf{T}^j}{h}=0$, therefore we only
need to study the Rees algebra of zero-dimensional ideals. Hence, without loss of generality, we consider ideals of the form $I=\langle T_0^{d_1},T_0^{u_1}T_1^{u_2},T_1^{d_2}\rangle$, where $d_1,u_1,d_2$ and $u_2$ are four integers.

\section{Algorithm for obtaining the minimal generating set of the Rees algebra of a monomial ideal in three generators}\label{sec:mainAlgorithm}
 
Let $I=\langle T_0^{d_1},T_0^{u_1}T_1^{u_2},T_1^{d_2}\rangle \subseteq \kb[T_0,T_1]$. In order to compute the generators of the Rees ideal of $I$ we need to find the essential binomials of the form $\mathbf{T}^\a\mathbf{X}^\b-\mathbf{T}^{\a'}\mathbf{X}^{\b'}$ with $\vert \b\vert=\vert \b'\vert=t$, for all $t$, such that $\phi(\mathbf{T}^{\a}\mathbf{X}^{\b}-\mathbf{T}^{\a'}\mathbf{X}^{\b'})=\psi(\mathbf{T}^{\a}\mathbf{X}^{\b}-\mathbf{T}^{\a'}\mathbf{X}^{\b'})=0$.
Lemma \ref{lemma:consecutive} and Propositions \ref{prop:general1} and \ref{prop:generalLimit} below give necessary conditions for essential binomials, while Proposition \ref{prop:general3} gives a sufficient condition. These propositions are constructive in the sense that we explicitly enumerate all the essential binomials, i.e. the minimal generators of $\Rk(I)$. This enumeration gives rise to Algorithm \ref{general:algorithm}, which is the main result of this section, see Theorem \ref{thm:general}.

First observe that since $T_1 \nmid \psi(X_0)$ and $T_0 \nmid \psi(X_2) $ only $\psi(X_1)$ can be expressed as a combination of $\psi(X_0)$ and $\psi(X_2)$, i.e. there is some tuple $(k,k_1,k_2) \in \mathbb{N}_0^3$ such that $\psi(X_1^k)=\psi(X_0^{k_1}X_2^{k_2})$. We get that the smallest integers satisfying that relation are
$$k=\lcm \left( \frac{d_1}{\gcd(d_1,u_1)},\frac{d_2}{\gcd(d_2,u_2)} \right)  \text{ ,  } k_1 =\frac{ku_1}{d_1} \text{ and }  k_2 =\frac{ku_2}{d_2}.$$
Let $G^{(t)}(I)=\{g^{(t)}_{1},\dots,g^{(t)}_{k_t} \}$ be the set of
products of $t$ elements of $G(I)$ ordered lexicographically. We denote by
$\psi'_t (g^{(t)}_i)$ the monomial $X_0^aX_1^bX_2^c\in S$ such that
$\psi(X_0^aX_1^bX_2^c)=g^{(t)}_i$ and $a+b+c=t$. Observe that this is
unique for $t<k$, and thus by a slight abuse of notation, we can consider
$\psi'_t$ and $\psi$ as mutually inverse for $t<k$. In the following
sections, we will say that a binomial $\mathbf{T}^\a\mathbf{X}^\b-\mathbf{T}^{\a'}\mathbf{X}^{\b'}$ in $S=\mathbb{K}[\mathbf{T}][\mathbf{X}]$ corresponds to the relation $s\cdot g^{(t)}_{i} - s'\cdot g^{(t)}_{j}$ in $\mathbb{K}[\mathbf{T}]$ if $\mathbf{T}^\a=s$, $\mathbf{T}^{\a'}=s'$, $\mathbf{X}^\b=\psi'_t(g_i^{(t)})$ and  $\mathbf{X}^{\b'}=\psi'_t(g_j^{(t)})$ and vice-versa.

\begin{lemma}\label{lemma:consecutive}
The essential binomials of $\Rk(I)$ correspond to relations of the form
$$s^{(t)}_{i,i+1}\cdot g^{(t)}_{i} - s'^{(t)}_{i,i+1}\cdot g^{(t)}_{i+1}, \text{ for }i=1,2,\dots ,k_t-1.$$
 where
 $$ s^{(t)}_{i,j}= \frac{\lcm(g^{(t)}_{i},g^{(t)}_{j})}{g^{(t)}_{i}} \text{  , and  }  s'^{(t)}_{i,j}= \frac{\lcm(g^{(t)}_{i},g^{(t)}_{j})}{g^{(t)}_{j}}, \text{ for } i<j.$$ 
\end{lemma}
\begin{proof}
Consider a binomial corresponding to the relation formed by two monomials that are not lexicographically consecutive in $G^{(t)}(I)$, i.e.
\[
s^{(t)}_{i,i+j}\cdot g^{(t)}_{i} - s'^{(t)}_{i,i+j}\cdot g^{(t)}_{i+j}, \text{ for some } j>1.
\]
Then,
\begin{align*}
s^{(t)}_{i,i+j}\cdot \psi'_t(g^{(t)}_{i}) - s'^{(t)}_{i,i+j}\cdot \psi'_t(g^{(t)}_{i+j})=&c_{i,i+1}\big( s^{(t)}_{i,i+1}\cdot \psi'_t(g^{(t)}_{i}) - s'^{(t)}_{i,i+1}\cdot \psi'_t(g^{(t)}_{i+1})\big)\\
&+c'_{i+1,i+j}\big(s^{(t)}_{i+1,i+j}\cdot \psi'_t(g^{(t)}_{i+1}) - s'^{(t)}_{i+1,i+j}\cdot \psi'_t(g^{(t)}_{i+j})\big),
\end{align*}
where $c_{i,i+1}=\frac{s^{(t)}_{i,i+j}}{s^{(t)}_{i,i+1}}=\frac{\lcm(g^{(t)}_i,g^{(t)}_{i+j})}{\lcm(g^{(t)}_i,g^{(t)}_{i+1})}$ and $c'_{i+1,i+j}=\frac{s'^{(t)}_{i,i+j}}{s'^{(t)}_{i+1,i+j}}=\frac{\lcm(g^{(t)}_i,g^{(t)}_{i+j})}{\lcm(g^{(t)}_{i+1},g^{(t)}_{i+j})}$.
Hence, the binomial $s^{(t)}_{i,i+j}\cdot \psi'_t(g^{(t)}_{i}) - s'^{(t)}_{i,i+j}\cdot \psi'_t(g^{(t)}_{i+j})$ is not a minimal generator.
\end{proof}

Observe that for $t=1$ we have two essential binomials, namely those corresponding to the only two minimal syzygies in $\Syz(I)$:
\begin{equation}\label{syz1}
    s^{(1)}_{1,2}\cdot g_{1}^{(1)}-s'^{(1)}_{1,2}\cdot g_{2}^{(1)} \mbox { corresponds to }  T_1^{u_2} X_0 - T_0^{d_1-u_1} X_1  \in G(\Rk(I))
\end{equation} 
\begin{equation}\label{syz2}
    s^{(1)}_{2,3}\cdot g_{2}^{(1)}-s'^{(1)}_{2,3}\cdot g_{3}^{(1)} \mbox { corresponds to } T_1^{d_2-u_2} X_1 - T_0^{u_1} X_2 \in G(\Rk(I)).
\end{equation}

As mentioned above, the minimal generating set of $\Rk(I)$ is composed of binomials of the form $\mathbf{T}^\a\mathbf{X}^\b-\mathbf{T}^{\a'}\mathbf{X}^{\b'}$ with $\vert \b\vert=\vert \b'\vert=t$, for all $t$, such that $\psi(\mathbf{T}^{\a}\mathbf{X}^{\b}-\mathbf{T}^{\a'}\mathbf{X}^{\b'})=0$. This implies that both terms in each essential binomial must not have any common factor, which can only be achieved if $\mathbf{X}^{\b}$ or $\mathbf{X}^{\b'}$ is actually a pure power of one of the $X_i$.
Note that for every $t>0$, $\psi(X_0^t)$ is the largest lexicographically ordered element in $G^{(t)}(I)$ and $\psi(X_2^t)$ is the smallest. By Lemma \ref{lemma:consecutive} we know that essential binomials come uniquely from the relations between two consecutive elements in $G^{(t)}(I)$; however, the only consecutive elements of $\psi(X_0^t)$ and  $\psi(X_2^t)$ are  $\psi(X_0^{t-1}X_1)$ and $\psi(X_1X_2^{t-1})$ respectively.
In both instances we have a common factor ($X_0^{t-1}$ and $X_2^{t-1}$ respectively); therefore, for $t>1$ there can not be any essential binomial involving the generators $X_0^t=\psi'_t(g^{(t)}_{1})$ and $X_2^t=\psi'_t(g^{(t)}_{k_t})$. Hence, we must only look for consecutive pairs in $G^{(t)}(I)$ involving the generator $g^{(t)}_{i}=\psi(X_1^t)$, because any potential essential binomial can only come from those. 
\begin{Notation}
    For the next results, we introduce the following notation:
\begin{itemize}
    \item[-]$q=\min{\bigl\{ \frac{d_1}{\gcd(d_1,u_1)}, \frac{d_2}{\gcd(d_2,u_2)}\bigr\}}$ and $[q]=\{1,\dots,q\}$.
    \item[-] $a^{(k)}_i= u_k\cdot i \pmod{d_k}$ and $b_i^{(k)}=\lfloor \frac{u_k\cdot i}{d_k}\rfloor $ for $i=1,2,\dots,q$ and $k=0,1$.
    \item[-] $\Delta_{\max}(d_k,u_k)=\{ i \in [q] \mid a_i^{(k)} =\max\{a^{(k)}_1,\dots,a^{(k)}_i\}\}$
    \item[-] $\Delta_{\min}(d_k,u_k)=\{ i \in [q] \mid a_i^{(k)}=\min\{a^{(k)}_1,\dots,a^{(k)}_i\}\}$
   
\end{itemize} 
\end{Notation}


\begin{lemma}\label{lem:ElementaryPropertiesegreeSet}
  $\Delta_{\max}(d_k,u_k)\cap \Delta_{\min}(d_k,u_k)=\{1\}$. Furthermore, for any $\delta\in\Delta_{\min}(d_k,u_k)$ and $\epsilon\in\Delta_{\max}(d_k,u_k)$, we have $a^{(k)}_{\delta}\leq a^{(k)}_{\epsilon}$ and this inequality is strict if $\delta\neq\epsilon$.
\end{lemma}
\begin{proof}
 This is clear from the definitions.
\end{proof}

\begin{lemma}\label{lem:MinMaxGeneralRelation}
 We have the following relations between $\Delta_{\min}(d_k,u_k)$ and $\Delta_{\max}(d_k,u_k)$:
 \begin{enumerate}
  \item Let $\delta\in\Delta_{\min}(d_k,u_k)$, and let $\epsilon=\min\{\gamma\in \Delta_{\min}(d_k,u_k)\mid \gamma>\delta\}$. Then, 
   \begin{enumerate}[label=(\alph*),nolistsep]
    \item  \label{lem:MinMaxGeneralRelation:item1a} $\epsilon-\delta\in \Delta_{\max}(d_k,u_k)$,
    \item \label{lem:MinMaxGeneralRelation:item1b} There is no element of $\Delta_{\max}(d_k,u_k)$ between $\epsilon-\delta$ and $\epsilon$.
   \end{enumerate}
  \item  Let $\delta\in\Delta_{\max}(d_k,u_k)$, and let $\epsilon=\min\{\gamma\in \Delta_{\max}(d_k,u_k)\mid \gamma>\delta\}$. Then, 
  \begin{enumerate}[label=(\alph*),nolistsep]
    \item \label{lem:MinMaxGeneralRelation:item2a} $\epsilon-\delta\in \Delta_{\min}(d_k,u_k)$,
    \item \label{lem:MinMaxGeneralRelation:item2b} There is no element of $\Delta_{\min}(d_k,u_k)$ between $\epsilon-\delta$ and $\epsilon$.
   \end{enumerate}
   \item In the situation of item 1, let $\zeta_1<\cdots<\zeta_r$ be the integers in $\Delta_{\max}(d_k,u_k)$ between $\delta$ and $\epsilon$. Then, for each $n$ with $1\leq n\leq r$,  we have:
    \begin{enumerate}[label=(\alph*),nolistsep]
     \item $\zeta_n-\delta\in\Delta_{\max}(d_k,u_k)$,
     \item There is no element of $\Delta_{\max}(d_k,u_k)$ between $\zeta_n-\delta$ and $\zeta_{n}$.
    \end{enumerate}
       \item In the situation of item 2, let $\zeta_1<\cdots<\zeta_r$ be the integers in $\Delta_{\min}(d_k,u_k)$ between $\delta$ and $\epsilon$. Then, for each $n$ with $1\leq n\leq r$,  we have:
    \begin{enumerate}[label=(\alph*),nolistsep]
     \item $\zeta_n-\delta\in\Delta_{\min}(d_k,u_k)$,
     \item There is no element of $\Delta_{\min}(d_k,u_k)$ between $\zeta_n-\delta$ and $\zeta_{n}$.
    \end{enumerate}
 \end{enumerate}

\end{lemma}
\begin{proof}
 We first prove items 1a and 2a. For this it suffices to prove item 1a, because item 2a is dual to it.
 
 Regarding item 1a, let $\delta\in \Delta_{\min}(d_k,u_k)$ and $\epsilon=\min\{\gamma\in \Delta_{\min}(d_k,u_k)\mid \gamma>\delta\}$. We know that $a^{(k)}_\delta=\min\{a^{(k)}_{i} \mid 1\leq i\leq\delta\}$ and $a^{(k)}_\epsilon=\min\{a^{(k)}_i\mid 1\leq i\leq\epsilon\}$. This implies in particular that $a^{(k)}_\delta>a^{(k)}_\epsilon\geq 0$. Moreover, there is no element of $\Delta_{\min}(d_k,u_k)$ between $\delta$ and $\epsilon$. Now we argue by \emph{reductio ad absurdum}. Assume that $\epsilon-\delta\notin \Delta_{\max}(d_k,u_k)$. Then there is an integer $1\leq j<\epsilon-\delta$ such that $a^{(k)}_j>a^{(k)}_{\epsilon-\delta}$. This implies
 $$a^{(k)}_\epsilon<a^{(k)}_{\delta+j}<a^{(k)}_\delta,$$
 but then $\Delta_{\min}(d_k,u_k)$ would contain an element between $\delta$ and $\epsilon$, a contradiction. Thus the assumption was false and $\epsilon-\delta\in \Delta_{\max}(d_k,u_k)$ as claimed.
 
 Now we prove items 1b and 2b. For this it suffices to prove item 1b, because item 2b is dual to it.
 
 Regarding item 1b, we argue by \emph{reductio ad absurdum}. Assume there were an integer $\zeta\in\Delta_{\max}(d_k,u_k)$ with $\epsilon-\delta<\zeta<\epsilon$, and take the minimal such $\zeta$. Then $\zeta=\min\{\gamma\in \Delta_{\max}(d_k,u_k)\mid \gamma>\epsilon-\delta\}$, and, using item 2a, we obtain $\eta:=\zeta-(\epsilon-\delta)\in\Delta_{\min}(d_k,u_k)$. Note that $\eta<\epsilon-(\epsilon-\delta)=\delta$. Moreover, $a^{(k)}_\zeta>a^{(k)}_{\epsilon-\delta}$; but this implies 
 $$a^{(k)}_\epsilon<a^{(k)}_{\delta+\zeta}<a^{(k)}_\delta,$$
 yielding $a^{(k)}_{\delta+\zeta-\epsilon}=a^{(k)}_\eta<a^{(k)}_\delta$,
 in contradiction to $\delta\in\Delta_{\min}(d_k,u_k)$. Hence the assumption was false and 1b holds.
 
Now we prove items 3a and 4a. For this it suffices to prove item 3a, because item 4a is dual to it.

Regarding item 3a, by hypothesis, we know that $a^{(k)}_\delta=\min\{a^{(k)}_i\mid 1\leq i\leq\delta\}$ and $a^{(k)}_{\zeta_n}=\max\{a^{(k)}_i\mid 1\leq i\leq\zeta_n\}$. This implies in particular that $0<a^{(k)}_\delta<a^{(k)}_{\zeta_n}<d$. Moreover, there is no element of $\Delta_{\min}(d_k,u_k)$ between $\delta$ and $\zeta_n$. Now we argue by \emph{reductio ad absurdum}. Assume that $\zeta_n-\delta\notin \Delta_{\max}(d_k,u_k)$. Then there is an integer $1\leq j<\zeta_n-\delta$ such that $a^{(k)}_j>a^{(k)}_{\zeta_n-\delta}$. This implies one of the following: Either,
 $$a^{(k)}_\delta<a^{(k)}_{\zeta_n}<a^{(k)}_{\delta+j}<d,$$
 in contradiction to $a^{(k)}_{\zeta_n}=\max\{a^{(k)}_i\mid 1\leq i\leq\zeta_n\}$; or,
 $$0\leq a^{(k)}_{\delta+j}<a^{(k)}_\delta<a^{(k)}_{\zeta_n},$$
 in which case 
 $\Delta_{\min}(d_k,u_k)$ would contain an element between $\delta$ and $\zeta_n$, again a contradiction. Thus, the assumption was false and $\zeta_n-\delta\in \Delta_{\max}(d_k,u_k)$ as claimed.
 
Now we prove items 3b and 4b. For this it suffices to prove item 3b, because item 4b is dual to it.

Regarding item 3b, we argue by \emph{reductio ad absurdum}. Assume that for some $n$ with $1\leq n\leq r$ there were an element $\eta_n\in\Delta_{\max}(d_k,u_k)$ with $\zeta_n-\delta<\eta_n<\zeta_n$, and take the minimal such $\eta_n$. Then $\eta_n=\min\{\gamma\in \Delta_{\max}(d_k,u_k)\mid \gamma>\zeta_n-\delta\}$, and, using item 2a, we obtain $\rho_n:=\eta_n-(\zeta_n-\delta)\in\Delta_{\min}(d_k,u_k)$. Note that $\rho_n<\zeta_n-(\zeta_n-\delta)=\delta$. Using item 2b, we get the additional information that there is no element of $\Delta_{\min}(d_k,u_k)$ between $\rho_n$ and $\eta_n$. But by hypothesis $\delta\in\Delta_{\min}(d_k,u_k)$, and $\rho_n<\delta$. Thus necessarily $\zeta_n-\delta<\eta_n<\delta$. As a consequence, $\zeta_1-\delta<\eta_n<\zeta_1$. Thus, from now on we may assume that $n=1$. We also write $\eta:=\eta_1$ and $\rho:=\rho_1$.

We aim to show that, under the made assumption, $r=1$. To this end, if $r>1$,  consider $\tilde{\eta}:=\zeta_2-\delta>\zeta_1-\delta$. By item 3a, $\tilde{\eta}\in\Delta_{\max}(d_k,u_k)$. There cannot be any element $\hat{\eta}\in \Delta_{\max}(d_k,u_k)$ between $\zeta_1-\delta$ and $\tilde{\eta}$, as this would induce an element of $\Delta_{\max}(d_k,u_k)$ between $\zeta_1$ and $\zeta_2$. Hence, $\tilde{\eta}=\eta$, and $\rho=(\zeta_2-\delta)-(\zeta_1-\delta)=\zeta_2-\zeta_1$. However, $\rho<\delta<\zeta_2$, and this contradicts item 2b applied to $\zeta_1$ and $\zeta_2$. Thus indeed we must have $r=1$.

Thus, under the made assumption, $\zeta:=\zeta_1$ is the only element in $\Delta_{\max}(d_k,u_k)$ between $\delta$ and $\epsilon$. 
By item 1a, $\epsilon-\delta\in\Delta_{\max}(d_k,u_k)$. By item 1b, there is no element of $\Delta_{\max}(d_k,u_k)$ between $\epsilon-\delta$ and $\epsilon$. This implies $\epsilon\geq \zeta+\delta$. Consider the integer $\sigma:=\zeta+\rho$. Since $\zeta<\sigma<\epsilon$, $\sigma\notin \Delta_{\min}(d_k,u_k)\cup\Delta_{\max}(d_k,u_k)$. Hence, $a^{(k)}_\delta<a^{(k)}_\sigma<a^{(k)}_\zeta$; otherwise, $a^{(k)}_\sigma$ would be a maximal or minimal value. Moreover, also $a^{(k)}_\rho$ lies between $a^{(k)}_\delta$ and $a^{(k)}_\zeta$: $a^{(k)}_\rho>a^{(k)}_\delta$, because $\{\rho,\delta\}\subseteq\Delta_{\min}(d_k,u_k)$ and $\rho<\delta$; $a^{(k)}_\rho<a^{(k)}_\zeta$, because $\rho\neq \zeta$ and $\zeta\in \Delta_{\max}(d_k,u_k)$. Furthermore, $a^{(k)}_\sigma<a^{(k)}_\rho$; otherwise, $a^{(k)}_\rho<a^{(k)}_{\zeta+\rho}<a^{(k)}_\zeta$, implying $a^{(k)}_\zeta<a^{(k)}_\zeta$, which is absurd. Thus, we have: $a^{(k)}_\delta<a^{(k)}_{\zeta+\rho}<a^{(k)}_\rho<a^{(k)}_\zeta$; and, as $\rho=\eta-\zeta+\delta$, we get $a^{(k)}_\delta<a^{(k)}_{\delta+\eta}<a^{(k)}_\rho$. This implies, finally, that $a^{(k)}_\eta<a^{(k)}_\rho$, the desired contradiction (note that $\eta\in\Delta_{\max}(d_k,u_k)$ and $\rho\in\Delta_{\min}(d_k,u_k)$). Hence the assumption was false and 3b holds.
\end{proof}
From parts (1) and (2) of this last lemma one can deduce the following:
\begin{corollary}\label{cor:nextrelevant}
    If $\delta \in \Delta_{\min}(d_k,u_k)$ where $\theta= \max \{\gamma \in \Delta_{\max}(d_k,u_k) \mid \gamma \leq \delta\}$, then we have that $\delta + \theta$ is either the index of the immediate next minimum after $\delta$ in $\Delta_{\min}(d_k,u_k)$ or the index of the immediate next maximum after $\theta$ in $\Delta_{\max}(d_k,u_k)$. Furthermore, $\nexists \lambda\in \Delta_{\max}(d_k,u_k)\cup \Delta_{\min}(d_k,u_k)$ such that $\delta < \lambda < \delta+\theta$.\\
    The same holds true for any  $\delta \in \Delta_{\max}(d_k,u_k)$ where $\theta= \max \{\gamma \in \Delta_{\min}(d_k,u_k) \mid \gamma \leq \delta\}$.
\end{corollary} 
Corollary \ref{cor:nextrelevant} tells us that, at any given point $i$ in the
sequence if we have that $a^{(k)}_\delta=\max\{a^{(k)}_1,\dots,a^{(k)}_i\}$ and
$a^{(k)}_\theta=\min\{a^{(k)}_1,\dots,a^{(k)}_i\}$, then either $\delta+\theta \in
\Delta_{\max}(d_k,u_k)$ or $\delta+\theta \in \Delta_{\min}(d_k,u_k)$ and there are
no elements between $i$ and $\delta+\theta$ in either $ \Delta_{\max}(d_k,u_k)$
or $ \Delta_{\min}(d_k,u_k)$. To put it simply, by adding the indices of the last maximum and the last minimum at a given point in  the sequence we get either the subindex of the next minimum or maximum in the sequence. 
\begin{lemma}\label{lem:ComplementarySeq}
    Let $i>1$ be the smallest integer which satisfies $b^{(1)}_i-b^{(1)}_{i-1}=b^{(2)}_i-b^{(2)}_{i-1}$. Then $ \Delta_{\min}(d_1,u_1)\cap [i-1]= \Delta_{\max}(d_2,u_2)\cap [i-1]$, and $ \Delta_{\max}(d_1,u_1)\cap [i-1]= \Delta_{\min}(d_2,u_2)\cap [i-1]$. In other words, for $j<i$ every time $a^{(1)}_j$ is a maximum on the sequence, $a^{(2)}_j$ is a minimum on the other one, and the other way around.
\end{lemma}
\begin{proof}
    Let us first assume $i>1$ is the smallest integer which satisfies $b^{(1)}_i-b^{(1)}_{i-1}=b^{(2)}_i-b^{(2)}_{i-1}$. By Lemma \ref{lem:ElementaryPropertiesegreeSet} note that $1$ is in all Deltas, since at $j=1$ we must have both a maximum and a minimum because there are no other elements. Hence, by Lemma \ref{lem:ComplementarySeq} at $1+1=2$ we must have either a maximum or a minimum in both sequences. If both $a^{(1)}_2$ and $a^{(2)}_2$ were maxima on their respective sequence, then we are done because $b^{(1)}_2-b^{(1)}_{1}=0$ and $b^{(2)}_2-b^{(2)}_{1}=0$; so, $i=2$. This is clear since by assumption $u_k<d_k$. Similarly, if both $a^{(1)}_2$ and $a^{(2)}_2$ were minima on their sequence $b^{(1)}_2-b^{(1)}_{1}=b^{(2)}_2-b^{(2)}_{1}=1$; so, $i=2$ and we are done. Thus, the only case left is if one is a maximum and the other a minimum, in which case we repeat the process by adding the positions of the "updated" last maxima and  minima. Note that we get the same outcome from these sums in both sequences since up until this point the sequences are complementary (in terms of maxima and minima) of each other, leading to the same scenario as described at $j=2$. 
    
\end{proof}

\begin{proposition}\label{prop:general1}
There are no essential binomials involving elements of $G^{(t)}(I)$ for $t>q$.  
\end{proposition}
\begin{proof}
At $t=q$ we have that $a^{(1)}_t=0$, and/or $a^{(2)}_t=0$ (depending on whether $q=d_1/\gcd(d_1,u_1)$ and/or $q=d_2/\gcd(d_2,u_2)$); thus, $\deg_{T_0}\psi(X_1^t)=\deg_{T_0}\psi(X_0^{b^{(1)}_t}X_2^{t-b^{(1)}_t})$, and/or $\deg_{T_1}\psi(X_1^t)=\deg_{T_1}\psi(X_0^{t-b^{(2)}_t}X_2^{b^{(2)}_t})$. Without loss of generality let us assume that $q=d_1/\gcd(d_1,u_1)$, which means that at least $a^{(1)}_q=0$ or equivalently  $\deg_{T_0}\psi(X_1^q)=u_1\cdot q=\deg_{T_0}(X_0^{b^{(1)}_q}X_2^{q-b^{(1)}_q})=d_1\cdot b^{(1)}_q$. This fact implies that:
\begin{subequations}
\begin{align}
   &  T_1^{\vert r\vert}X_1^q - X_0^{b^{(1)}_q}X_2^{q-b^{(1)}_q} \in \Rk(I)  \text{  or  }\label{syz_proof_a} \\
   &  T_1^{\vert r\vert}X_0^{b^{(1)}_q}X_2^{q-b^{(1)}_q} - X_1^q  \in  \Rk(I) \label{syz_proof_b} ,
\end{align}
\end{subequations}
for some $\vert r\vert=\vert d_2(q-b^{(1)}_q)-u_2q \vert$. If $r\geq 0$ we are in the first case, otherwise we are in the second case.

Let now $t>q$ i.e. $t=q+j$ for some $j \in \mathbb{Z}^+$. From Lemma \ref{lemma:consecutive} we know that any essential binomial whose monomials correspond to elements of $G^{(t)}(I)$ must come from some lexicographically consecutive pair of elements of $G^{(t)}(I)$, we also know that one of them must correspond to the pure power $t$ of $X_1$ and the other one to some combination of only $X_0$ and $X_2$. Since $d_1\cdot (b^{(1)}_{q+j-1}+1) \geq u_1 \cdot (q+j) \geq d_1 \cdot b^{(1)}_{q+j-1}$ we know that the generator lexicographically closest to $\psi(X_1^{q+j})$ satisfying the previous condition are $\psi \big( X_0^{b^{(1)}_{q+j-1}+1}X_2^{q+j-(b^{(1)}_{q+j-1}+1)} \big)$ and $\psi \big( X_0^{b^{(1)}_{q+j-1}}X_2^{q+j-(b^{(1)}_{q+j-1})} \big)$ yielding the following  potentially essential binomials:
\begin{align*}
  &T_1^{r_1}X_0^{b^{(1)}_{q+j-1}+1}X_2^{q+j-(b^{(1)}_{q+j-1}+1)} -  T_0^{r_2}\cdot X_1^{q+j},\text{ and }\\
  &T_1^{r'_1}\cdot X_1^{q+j} - T_0^{r'_2}X_0^{b^{(1)}_{q+j-1}}X_2^{q+j-b^{(1)}_{q+j-1}}, 
\end{align*}

where $r_1 = u_2\cdot (q+j)- d_2 \cdot \big(q+j-(b^{(1)}_{q+j-1}+1)\big)$, $r_2=d_1\cdot (b^{(1)}_{q+j-1}+1)-u_1 \cdot (q+j)$, $r'_1=d_2\cdot \big( q+j -(b^{(1)}_{q+j-1}) \big) -u_2 \cdot (q+j)$, and $r'_2=u_1\cdot (q+j)- d_1 \cdot (b^{(1)}_{q+j-1})$.

Note that if for $t=q$ we have that the binomial (\ref{syz_proof_b}), i.e. $T_1^{\vert r\vert} X_0^{b^{(1)}_q}X_2^{q-b^{(1)}_q} - X_1^q$, is in $\Rk(I)$, then one can express the term containing the pure power $X_1^{q+j}$ as a combination of all  $X_0,X_1$ and $X_2$ in both of the potentially essential binomials, henceforth guaranteeing a common factor between the two terms in both binomials. On the other hand, if we have that the element (\ref{syz_proof_a}), i.e $\Big( T_1^{r}\cdot X_1^q - X_0^{b^{(1)}_q}X_2^{q-b^{(1)}_q} \Big)$ is in $\Rk(I)$, then we need to see that $X_0^{b^{(1)}_q}X_2^{q-b^{(1)}_q}$ divides both $X_0^{b^{(1)}_{q+j-1}+1}X_2^{q+j-(b^{(1)}_{q+j-1}+1)}$ and $X_0^{b^{(1)}_{q+j-1}}X_2^{q+j-b^{(1)}_{q+j-1}}$, in order to demonstrate that there exists a common factor $X_1$  between the terms in the binomials. In other words, we need to show that $b^{(1)}_q \leq b^{(1)}_{q+j-1}$ and $q-b^{(1)}_q \leq q+j-(b^{(1)}_{q+j-1}+1)$. The first inequality is pretty straightforward since $b^{(1)}_\alpha \geq b^{(1)}_\beta$ for any $\alpha \geq \beta$ by definition, hence, $b^{(1)}_q \leq b^{(1)}_{q+j-1}$ since we assumed $j\in \mathbb{Z}^+$. For the second inequality, note that one can easily deduce from the definitions that $\alpha - \beta \geq b^{(1)}_\alpha - b^{(1)}_\beta$ for any two $\alpha,\beta \in \mathbb{Z}^+_0$ such that $\alpha \geq \beta$. Since $q+j-1\geq q$ and, as shown above, $b^{(1)}_{q+j-1} \geq b^{(1)}_q$, then we can write the following inequality:
\begin{align*}
    & (q+j-1)- q \geq b^{(1)}_{q+j-1} -b^{(1)}_q \\
    \hspace{-1.5cm} \Leftrightarrow \hspace{1cm} & \hspace{1.46cm} b^{(1)}_q - q \geq b^{(1)}_{q+j-1} - (q+j) + 1  \\
    \hspace{-1.5cm} \Leftrightarrow \hspace{1cm} & \hspace{1.46cm} q-b^{(1)}_q \leq q+j-(b^{(1)}_{q+j-1}+1)
\end{align*}

This concludes the proof, since it shows that the only two possible binomials whose monomials correspond to elements in $G^{(t)}(I)$ for $t>q$ that could potentially be essential are indeed non-essential.
\end{proof}
\begin{proposition}\label{prop:generalLimit}
    If $b^{(1)}_i-b^{(1)}_{i-1}=b^{(2)}_i-b^{(2)}_{i-1}$ for some integer $i$, then there do not exist any essential binomials involving elements of $G^{(t)}(I)$ for any $t>i$.
\end{proposition}
\begin{proof}
Let $i>1$ be the smallest integer such that $b^{(1)}_i-b^{(1)}_{i-1}=b^{(2)}_i-b^{(2)}_{i-1}$. For each $j<i$ we have that either $b^{(1)}_j=b^{(1)}_{j-1}+1$ or $b^{(2)}_j=b^{(2)}_{j-1}+1$, and since $b^{(1)}_1=b^{(2)}_1=0$ then $b^{(1)}_j+b^{(2)}_j=j-1$ for every $j<i$. Hence it is straightforward to see that $\psi \big( X_0^{b^{(1)}_j +1}X_2^{b^{(2)}_j}\big)$ and $\psi\big( X_0^{b^{(1)}_j}X_2^{b^{(2)}_j+1}\big)$ are the two elements of $G^{(j)}(I)$ that can be written as a combination of powers of $X_0$ and $X_2$, which are lexicographically closest to $\psi \big(X_1^j\big)$, for $1<j<i$. This points to the only 2 pairs of monomials in $G^{(j)}(I)$ that can form a candidate for the essential binomials.

Now, for $G^{(i)}$, consider the two possible cases:
   \begin{enumerate}[i)]
        \item {
            Let $b^{(1)}_i-b^{(1)}_{i-1}=b^{(2)}_i-b^{(2)}_{i-1}=0$, then we have that $\psi \big(X_1^i\big)$ is the next lexicographically ordered element after $\psi \big( X_0^{b^{(1)}_i +1}X_2^{b^{(2)}_i+1}\big)$, where:
            $$\deg_{T_0}\big( \psi \big( X_0^{b^{(1)}_i +1}X_2^{b^{(2)}_i+1}\big)\big) > \deg_{T_0}\big( \psi\big( X_1^i\big)\big)$$ and $$\deg_{T_1}\big( \psi \big( X_0^{b^{(1)}_i +1}X_2^{b^{(2)}_i+1}\big)\big) > \deg_{T_1}\big( \psi \big( X_1^i\big)\big).$$
            Hence, we have that the following binomial is in $\Rk(I)$:
             \begin{equation}\label{rel_syz1}
             T_0^{d_1 (b^{(1)}_i +1)-u_1 i}T_1^{ d_2 (b^{(2)}_i +1)-u_2 i } X_1^i - X_0^{b^{(1)}_i+1}X_2^{b^{(2)}_i+1} .
             \end{equation}
             }
 
  \item Let $b^{(1)}_i-b^{(1)}_{i-1}=b^{(2)}_i-b^{(2)}_{i-1}=1$,  then we have that $\psi \big( X_0^{b^{(1)}_i +1}X_2^{b^{(2)}_i+1}\big)$ is the next lexicographically ordered element after $\psi \big(X_1^i\big)$, where:
        $$\deg_{T_0}\big( \psi \big( X_0^{b^{(1)}_i }X_2^{b^{(2)}_i}\big)\big) \leq \deg_{T_0}\big( \psi \big( X_1^i\big)\big)$$ and $$\deg_{T_1}\big( \psi \big( X_0^{b^{(1)}_i }X_2^{b^{(2)}_i}\big)\big) \leq \deg_{T_1}\big( \psi \big( X_1^i\big)\big).$$
    Hence, the following binomial is in $\Rk(I)$:
    \begin{equation}\label{rel_syz2}
         T_0^{u_1 i-d_1 b^{(1)}_i}T_1^{ u_2 i - d_2 b^{(2)}_i } X_0^{b^{(1)}_i}X_2^{b^{(2)}_i} -X_1^i.
    \end{equation}
   \end{enumerate}

The rest of the proof follows the lines of the proof of Proposition \ref{prop:general1}. There are no essential binomials in $G^{(t)}(I)$, $t>i$ since any relation between elements in $G^{(t)}(I)$ for $t>i$ can be expressed in terms of either (\ref{rel_syz1}) or (\ref{rel_syz2}).

\end{proof}
\begin{proposition}\label{prop:general3}
    Let $i$ be the smallest integer which satisfies  $b^{(1)}_i-b^{(1)}_{i-1}=b^{(2)}_i-b^{(2)}_{i-1}$ and let $j < \min\{i,q\}$.
    
    There exists an essential binomial in $G^{(j)}(I)$ of the form: 
    $$ T_1^{a^{(2)}_j}\psi(X_0^{b^{(1)}_j+1}X_2^{b^{(2)}_j})-T_0^{d_1-a^{(1)}_j}\psi(X_1^j)$$
    if and only if $a^{(1)}_j= \max \{ a^{(1)}_k | k=1,\dots,j\}$ and $a^{(2)}_j= \min \{ a^{(2)}_k | k=1,\dots,j\}$. \\
    Similarly, there exists an essential binomial in $G^{(j)}(I)$ of the form:
    $$T_1^{d_2-a^{(2)}_j}\psi(X_1^j) - T_0^{a^{(1)}_j}\psi(X_0^{b^{(1)}_j}X_2^{b^{(2)}_j+1})$$
    if and only if $a^{(1)}_j= \min \{ a^{(1)}_k | k=1,\dots,j\}$ and $a^{(2)}_j= \max \{ a^{(2)}_k | k=1,\dots,j\}$.
    
\end{proposition}
\begin{proof}
    Note that due to Lemma \ref{lem:ComplementarySeq} we know that for all $j<i$, $a^{(2)}_j$ is a minimum on its sequence whenever $a^{(1)}_j$ is a maximum on the other one, and similarly $a^{(2)}_j$ is a maximum whenever $a^{(1)}_j$ is a minimum; and the other way around. Since we assumed $j<\min\{q,i\}\leq i$, we only need show the statements hold for one of the sequences.

    Let us consider the case where $a^{(1)}_j$ is a maximum for some $j < \min\{i,q\}$, the case that $a^{(1)}_j$ is a minimum is analogous. 

    As we have already mentioned, an essential binomial comes from a relation of a consecutive pair of monomials in $G^{(j)}(I)$ where one of the two is the monomial $h=\psi(X_1^j)=T_0^{u_1\cdot j}\cdot T_1^{u_2\cdot j}$.\\
    Since  $a^{(1)}_j= u_1\cdot i \pmod{d_1}$, we have that $u_1\cdot j = d_1 \cdot k + a^{(1)}_j$ for some $k$. Hence $k=b^{(1)}_j$, and $\deg_{T_0}(h)=d_1 \cdot b^{(1)}_j + a^{(1)}_j$. Correspondingly, $\deg_{T_1}(h)=d_2\cdot b^{(2)}_j+a^{(2)}_j$. Since we know that $b^{(1)}_j+b^{(2)}_j=j-1$ then $b^{(2)}_j=j-(b^{(1)}_j+1)$, and $\deg_{T_1}(h)=d_2\cdot(j-(b^{(1)}_j+1))+a^{(2)}_j$.
    
   Let $h'$ be the immediate previous element in $G(I^j)$ i.e. $h' \prec_{lex}h$, and $\nexists \hat{h} \in G^{(j)}(I)$ such that $h' \prec_{lex} \hat{h} \prec_{lex}h$. Recall that $\psi'_t(h')= X_0^aX_1^bX_2^c$ where $(a,b,c)$ is the unique 3-tuple such that $a+b+c=j$. Observe that the essential binomial coming from the pair $(h',h)$ is relevant if $b=0$. Thus, we have that $\deg_{T_0}(h')=d_1\cdot a$ and so we get:
    $$d_1 \cdot b^{(1)}_j + a^{(1)}_j < d_1\cdot a$$
    Note that $b^{(1)}_j \leq \frac{u_1}{d_1}\cdot j$, and by assumption $u_1<d_1$; hence, we have that $b^{(1)}_j<j$. Thus, there exists an element $g \in G^{(j)}(I)$ corresponding to the 3-tuple $\big(b^{(1)}_j+1,0,j-(b^{(1)}_j+1) \big)=\big(b^{(1)}_j+1,0,b^{(2)}_j)$ with $\deg_{T_0}(g)=d_1 \cdot b^{(1)}_j + d_1$, which implies $g\prec_{lex} h$. We must show that $g=h'$, i.e.  $\nexists \hat{h} \in G^{(j)}(I)$ such that $\deg_{T_0}(g) >\deg_{T_0}( \hat{h}) >\deg_{T_0}(h)$ if and only if $a^{(1)}_j = \max\{a^{(1)}_k\vert k=1,\dots, j\}$. In order to get a contradiction on the first direction let us assume that there exists such an element $\hat{h}\in G(I^j)$, with $\psi'_t(\hat{h})= X_0^{\hat{a}}X_1^{\hat{b}}X_2^{\hat{c}}$. Then, we have that:
    $$d_1 \cdot b^{(1)}_j + a^{(1)}_j < d_1\cdot \hat{a}+u_1\cdot \hat{b} < d_1 \cdot b^{(1)}_j + d_1$$
    for some $\hat{a},\hat{b}\geq 0$ such that $\hat{a}+\hat{b}\leq j$.
    However, if we subtract $d_1\cdot b^{(1)}_j$ from every side of the inequality we get that:
    $$a^{(1)}_j< d_1\cdot (\hat{a}-b^{(1)}_j)+u_1\cdot \hat{b}< d_1.$$
    Since the difference between $a^{(1)}_j$ and $d_1$ is by definition smaller that $d_1$ this implies $\hat{a}=b^{(1)}_j$ and $u_1\cdot\hat{b} \pmod{d_1} >a^{(1)}_j$ where $\hat{b}<j$ (because $\hat{b}=j$ implies that $h=\hat{h}$). This contradicts our assumption that $a^{(1)}_j = \max\{a^{(1)}_k\vert k=1,\dots, j\}$.
    
    For the other direction, let's assume there is some $p<j$ with $d_1>a^{(1)}_p>a^{(1)}_j$. By adding $d_1\cdot b^{(1)}_j$ to the inequality we get
    $$d_1 \cdot b^{(1)}_j + a^{(1)}_j < d_1\cdot b^{(1)}_j +a^{(1)}_p < d_1 \cdot b^{(1)}_j + d_1$$
    which can be rewritten, using that $a^{(1)}_p=u_1\cdot p -d_1 \cdot b^{(1)}_p$, as:
    $$d_1 \cdot b^{(1)}_j + a^{(1)}_j < d_1\cdot (b^{(1)}_j-b^{(1)}_p) + u_1\cdot p < d_1 \cdot b^{(1)}_j + d_1$$
    which implies that exists an element $\hat{h}\in G^{(j)}(I)$ such that 
    \begin{displaymath}
        \psi'_j(\hat{h})= X_0^{j-p-(b^{(1)}_j-b^{(1)}_p) } X_1^p X_2^{(b^{(1)}_j-b^{(1)}_p)}\,.
    \end{displaymath}
    This element sits lexicographically between $h$ and $h'$ and thus, by Lemma \ref{lemma:consecutive}, there does not exist such an essential binomial in $\Rk(I)$ corresponding to the pair $(h',h)$.
    
    We can then claim  that there does not exist any element in $G^{(j)}(I)$, located between $\psi \big( X_0^{b^{(1)}_j +1}X_2^{b^{(2)}_j}\big)$  and $\psi\big(X_1^j\big)$ in the lexicographic order if and only if $a^{(1)}_j= \max \{ a^{(1)}_k | k=1,\dots,j\}$ and $a^{(2)}_j= \min \{ a^{(2)}_k | k=1,\dots,j\}$. In other words we have the lexicographically consecutive pair $\Big( \psi \big( X_0^{b^{(1)}_j +1}X_2^{b^{(2)}_j}\big) , \psi \big(X_1^j\big) \Big)$ which yields the essential binomial corresponding to $T_1^{a^{(2)}_j}\psi(X_0^{b^{(1)}_j+1}X_2^{b^{(2)}_j})-T_0^{d_1-a^{(1)}_j} \psi(X_1^j) \in \Syz(I^j)$  for every $j < \min\{i,q\}$ if and only if $a^{(1)}_j= \max \{ a^{(1)}_k | k=1,\dots,j\}$ and $a^{(2)}_j= \min \{ a^{(2)}_k | k=1,\dots,j\}$. 
\end{proof}

The previous conditions are used in Algorithm \ref{general:algorithm} to compute the minimal generating set of the Rees ideal of any monomial ideal in two variables with three minimal generators.

\begin{algorithm}
\caption{Minimal generating set of Rees ideal $\Rk(I)$ where $I$ is a monomial ideal in the polynomial ring in two variables generated by any three minimal generators}\label{general:algorithm}
\begin{algorithmic}[1]
\Require Three minimal generators $g_1,g_2$ and $g_3$.
\Ensure Minimal generating set of $\Rk(I)$
\vspace{0.2cm}
\State{get\_param($g_1,g_2,g_3$)} \Comment{subalgorithm retrieving $u_1,u_2,d_1,$ and $d_2$ from $\{g_1,g_2,g_3\}$}
\State $a^{(1)}\gets u_1, a^{(2)}\gets u_2, b^{(1)}\gets 0, b^{(2)}\gets 0, b^{(1)}_2\gets 0, b^{(2)}_2\gets 0,j\gets 1$;
\State $gens\gets \{\}$;
\State $min^{(1)}\gets a^{(1)}, min^{(2)}\gets a^{(2)}$
\State $max^{(1)}\gets a^{(1)}, max^{(2)}\gets a^{(2)}$
\State $gens\gets gens \cup \{ T_1^{u_2} X_0-T_0^{d_1-u_1} X_1\} $
\State $gens\gets gens \cup \{T_1^{d_2-u_2} X_1 - T_0^{u_1} X_2\} $
\While{$a^{(1)} \neq 0$ and $a^{(2)} \neq 0$}
\State{$j\gets j+1$}
\State $a^{(1)} \gets (u_1*j) \mod d_1$
\State $a^{(2)}\gets (u_2*j) \mod d_2$
\State $b^{(1)}\gets u_1*j/d_1$;
\State $b^{(2)}\gets u_2*j/d_2$;
\If {$b^{(1)}-b^{(1)}_2==b^{(1)}-b^{(1)}_2$}
    \If {$b^{(1)}-b^{(1)}_2==0$}
    \State $gens\gets gens \cup  \{T_0^{d_1-a^{(1)}}T_1^{d_2-a^{(2)}} X_1^j-X_0^{b^{(1)}+1}X_2^{b^{(2)}+1}\} $
    \Break
    \EndIf
    \If {$b^{(1)}-b^{(1)}_2==1$}
    \State $gens \gets gens \cup \{T_0^{a^{(1)}}T_1^{a^{(2)}} X_0^{b^{(1)}}X_2^{b^{(2)}}-X_1^j\}$
    \Break
    \EndIf
\Else
    \If{$(a^{(1)} > max^{(1)} \mbox{ \bf{and} }a^{(2)}<min^{(2)})$}
    \State $gens \gets gens \cup  \{T_1^{a^{(2)}}X_0^{b^{(1)}+1}X_2^{b^{(2)}}-T_0^{d_1-a^{(1)}} X_1^j\}$
    \State $\mbox{\bf update } max^{(1)},min^{(2)}$
    \EndIf
    \If{$(a^{(1)} < min^{(1)} \mbox{ \bf{and} }a^{(2)}>max^{(2)})$}
    \State $gens \gets gens \cup \{T_1^{d_2-a^{(2)}}X_1^j - T_0^{a^{(1)}} X_0^{b^{(1)}}X_2^{b^{(2)}+1}\}$
    \State $\mbox{\bf update }min^{(1)}, max^{(2)}$
    \EndIf
\State $b^{(1)}_2\gets b^{(1)}$
\State $b^{(2)}_2\gets b^{(2)}$
\EndIf
\EndWhile\\
\Return $gens$
\end{algorithmic}
\end{algorithm}

\begin{theorem}\label{thm:general}
Let $I=\langle T_0^{d_1},T_0^{u_1}T_1^{u_2},T_1^{d_2}\rangle\subseteq R=\kb[T_0,T_1]$ a monomial ideal in two variables generated by three monomials, and let $q=\min\{ d_1/ \gcd(d_1,u_1), d_2/\gcd(d_2,u_2)\}$. Algorithm \ref{general:algorithm} terminates in at most $q$ steps and returns the minimal generating set of the Rees ideal of $I$.
\end{theorem}

\begin{proof}
As we saw above, the minimal generating set of $\Rk(I)$, is made of binomials that correspond to relations between pairs of products of powers of the  same degree of $g_0 = T_0^{d_1},g_1=T_0^{u_1}T_1^{u_2}$, and $g_2=T_1^{d_2}$ that can not be expressed as any combination of other such relations. Let us consider the lexicographically ordered set of all these products of powers of degree $t$, then Lemma \ref{lemma:consecutive} tells us that all the relations between any pair that is not made of consecutive elements in the set can never correspond to a minimal generator for all $t\in \mathbb{N}^+$. Excluding all those and given the fact that one of the two elements in the pair must be a pure power $g_1^t$ as claimed at the beginning of the section; that yields the conclusion that the only minimal generators are the binomials that correspond to the relation of this pure power and a product of powers of $g_0$ and $g_2$ that is located immediately in front or prior to $g_1^t$ in the ordered set.
 
 Let $i$ be the smallest integer which satisfies  $b^{(1)}_i-b^{(1)}_{i-1}=b^{(2)}_i-b^{(2)}_{i-1}$. For $t<\min\{q,i\}$ Proposition \ref{prop:general3} finds all such binomials. Then, by Propositions \ref{prop:general1} and \ref{prop:generalLimit} at $t=q$ and at $t=i$, we expect to find some relation of divisibility between $g_1^t$ and $g_0^jg_2^{t-j}$, for some integer $j<t$. Whichever of these comes from the smaller degree is the one corresponding to a binomial that is a minimal generator. In fact, this is the last element in the minimal generating set since due to these last two Propositions we know that all other relations after this one can be expressed as combinations of the ones we have already found. This shows the correctness and termination of Algorithm \ref{general:algorithm}.
\end{proof}
Note that since both terms in every binomial generator of the Rees ideal of $I$ have $+/-1$ constant coefficients we will refer to the minimal generating set as the unique normalized set of binomials up to a possible sign change of the elements.

\begin{example}\label{Ex1}
    Let $I=\langle T_0^{15},T_0^9T_1^6,T_1^{13}\rangle$. Both $q=5$ and $i=5$, so Algorithm \ref{general:algorithm} finishes in 5 steps and returns the minimal generating set of $\Rk(I)$. First, the first two elements are added to the set of minimal generators of $\Rk(I)$ as seen in (\ref{syz1}) and (\ref{syz2}). They correspond to  each of the two minimal first syzygies of $I^1$, namely:
    $$g_1= T_1^6X_0-T_0^6X_1 \text{ and } g_2= T_1^7X_1-T_0^9X_2$$
     The algorithm proceeds as follows: at $j=2$ we have a min from $a^{(1)}_2$ and (as expected by Lemma \ref{lem:ComplementarySeq}) a max from $a^{(2)}_2$; hence, the essential binomial $g_3=T_1X_1^2-T_0^3X_0X_2$ is added. At $j=3$ we get a max from $a^{(1)}_3$ and a min from $a^{(2)}_3$, so  the essential binomial $g_4=T_1^5X_0^2X_2-T_0^3X_1^3$ is added. At $j=4$ nothing relevant is found as expected from Corollary \ref{cor:nextrelevant}. Finally, at $j=5=\min\{i,q\}$ the Algorithm detects the divisibility relation, the  corresponding binomial $g_5=T_1^4X_0^3X_2^2-X_1^5$ is added, and the Algorithm terminates.
     
     Table \ref{table:example1} shows the trace of the variables across the steps of the algorithm for this example. Columns $\max$ and $\min$ only show numbers when they are actually updated and therefore new elements are added to the set of minimal generators of $\Rk(I)$.

    \begin{table}
    \begin{tabular}{c|cccccccc|l}
        $j$&$a^{(1)}$&$b^{(1)}$&$a^{(2)}$&$b^{(2)}$&$max^{(1)}$&$max^{(2)}$&$min^{(1)}$&$min^{(2)}$&$gens$\\
        \hline
        $1$&$9$&$0$&$6$&$0$&$9$&$6$&$9$&$6$&$T_1^6X_0-T_0^6X_1, T_1^7X_1-T_0^9X_2$ \\
        $2$&$3$&$1$&$12$&$0$&&$12$&$3$&&$T_1X_1^2-T_0^3X_0X_2$ \\
        $3$&$12$&$1$&$5$&$1$&$12$&&&$5$&$T_1^5X_0^2X_2-T_0^3X_1^3$ \\
        $4$&$6$&$2$&$11$&$1$&&&&& \\
        $5$&$0$&$3$&$4$&$2$&&&$0$&$4$&$T_1^4X_0^3X_2^2-X_1^5$ \\
        \hline
    \end{tabular}
    \caption{Trace of Algorithm \ref{general:algorithm} for $I=\langle T_0^{15},T_0^9T_1^{6},T_1^{13}\rangle$.}\label{table:example1}
\end{table}
\end{example}

\begin{Note}
Corollary \ref{cor:nextrelevant} tells us that if one adds the indices of the last maximum and minimum at any given point in one of the sequences then one gets the index of the immediately next relevant element (a maximum or a minimum) in the sequence. By repeating this process with the upgraded indices of the maximum and minimum then one can compute all the indices where minimal generators of the Rees Ideal are added, without actually having to go through all indices. This can save some computations in Algorithm \ref{general:algorithm} but we have opted not to include this optimisation in order to ease the description of the algorithm. 
To make use of such optimisation, one would need to add two more variables initialised to $1$, for example: $c1 \leftarrow 1$ and $c2 \leftarrow 1$. Then, line $9$ in Algorithm \ref{general:algorithm} would need to be changed to $j \leftarrow c1+c2$, and last: the new indices need to be updated so lines $26$ and $30$ should update also the variables $c1$ and $c2$ respectively, i.e. by setting $c1 \leftarrow j$ and $c2 \leftarrow j$ respectively.
\end{Note}

\section{Rees algebras of monomial plane curve parametrizations}\label{sec:monomialCurveParametrization}

A particularly interesting case of application of Theorem \ref{thm:general} is to the ideal of a parametrization of a plane monomial curve, see \cite{CortadellasDAndrea}. In this case, the ideal under consideration is of the form $I=\langle T_0^d, T_0^{u}T_1^{d-u},T_1^d\rangle\subset R=\Kbb[T_0,T_1]$, i.e.\ $d_1=d_2=d$ and $u_2=d-u_1$. These ideals correspond to the equigenerated case of the ideals studied in the previous section. Because of that we can describe their minimal generating set using just two parameters. As a result, we are able to adapt the propositions in the preceding sections to a simpler form and hence to give an adapted version of Algorithm \ref{particular:algorithm} for this particular case.

\begin{proposition}[Adaptation of Proposition \ref{prop:general1} and Proposition \ref{prop:generalLimit}]\label{prop:particular1}
Let $t=\frac{d}{\gcd(d,u)}$. There are no essential binomials involving elements of $G^{(j)}(I)$ for $j>t$.  
\end{proposition}
\begin{proof}
    The proof of Proposition \ref{prop:general1} can be easily adapted to this case since $\gcd(d,u)=\gcd(d,d-u)$ which yields $t=q$ and as a result we obtain that the binomial 
\[
X_1^{\frac{d}{\gcd(d,u)}}-X_0^{\frac{u}{\gcd(d,u)}}X_2^{\frac{d-u}{\gcd(d,u)}}
\]
is an essential binomial for $\Rk(I)$ and it comes from the one relation of divisibility between elements in $G^{(t)}(I)$.

To see that $i=t$ is the smallest integer such that $b^{(1)}_i-b^{(1)}_{i-1}=b^{(2)}_i-b^{(2)}_{i-1}$, is slightly less straightforward. Remember that we saw in the previous section that $b^{(1)}_j-b^{(1)}_{j-1}\neq b^{(2)}_j-b^{(2)}_{j-1}$ is equivalent to  saying that $b^{(1)}_j+b^{(2)}_j= j-1$ for every $j<t$. 

Note that $u_1\cdot i = u\cdot i = d\cdot k + a^{(1)}_j$  for some integer $k$, where $0 < a^{(1)}_j <d$ for every $1\leq j < t$ by definition. Then,

\begin{align*}
    u_2 \cdot j &= (d-u)\cdot j \\
    \hspace{-1.5cm} \Leftrightarrow \hspace{1cm} & = d\cdot j - (d\cdot k +a^{(1)}_j)  \\
    \hspace{-1.5cm} \Leftrightarrow \hspace{1cm} & = d(j-k-1)+d-a^{(1)}_j
\end{align*}
Since $0 < d-a^{(1)}_j < d$, for every $1\leq j < t$; we then have that $b^{(1)}_j+b^{(2)}_j=k+(j-k-1)=j-1$, as desired. Note that also at $j=t$ we have $a^{(1)}_t=0$ and thus
$b^{(1)}_t+b^{(2)}_t=k+(t-k)=t$. But, $b^{(1)}_{t-1}+b^{(2)}_{t-1}=k+((t-1)-k-1)=t-2$, which implies as desired that $i=t$ is the smallest integer such that $b^{(1)}_i-b^{(1)}_{i-1}=b^{(2)}_i-b^{(2)}_{i-1}=1$.
\end{proof}

Therefore, for the particular case of the ideal of the parametrization, it
does not make much sense to continue using two sequences since each one
contains the same information as the other one and we just saw that the stopping step of the algorithms is always $t=\frac{d}{\gcd(d,u)}$. 

For integers below $t$, the essential binomials are described by the following adaptation of Proposition \ref{prop:general3}.

\begin{proposition}[Adaptation of Proposition \ref{prop:general3}]\label{prop:particular3}
Let $a_i= u\cdot i \pmod{d}$ and $b_i=\lfloor \frac{u\cdot i}{d}\rfloor $ for $i=1,2,\dots,\frac{d}{\gcd(d,u)}$. Then, there exists an essential binomial formed by elements of $G^{(j)}(I)$ of the form:
$$T_1^{d-a_j}\cdot \psi(X_0^{b_j+1}X_2^{j-(b_j+1)})-T_0^{d-a_j}\cdot \psi(X_1^j)$$
if and only if $a_j= \max \{ a_i | i=1,\dots,j\}$ for every $j\leq \frac{d}{\gcd(d,u)}$.

Similarly, there exists an essential binomial formed by elements of $G^{(j)}(I)$ of the form:
$$T_1^{a_j}\cdot \psi(X_1^j) - T_0^{a_j}\psi(X_0^{b_j}X_2^{j-b_j})$$
if and only if $a_j= \min \{ a_i | i=1,\dots,j\}$ for every $j\leq \frac{d}{\gcd(d,u)}$.    
\end{proposition}

As a result of the last two propositions we are able to adapt Algorithm \ref{general:algorithm} for the case of parametrizations of monomial plane curves obtaining a quite simpler version in the form of Algorithm \ref{particular:algorithm}.

\begin{algorithm}
 \caption{Minimal generating set of Rees ideal $\Rk(I)$ where $I$ is the associated ideal to some monomial plane curve parametrization}\label{particular:algorithm}
 \begin{algorithmic}[1]
 \Require Parameters $d$ and $u$ of the monomial ideal defining the monomial plane curve.
 \Ensure Minimal generating set of $\Rk(I)$
 \State $gens\gets\{\}$
 \State $j\gets 1,\, a \gets u, \, b \gets 0$
 \State $min \gets a$
 \State $max \gets a$
 \State $gens \gets gens\cup \left\{ T_1^{d-u} X_0-T_0^{d-u} X_1\right\} $
 \State $gens \gets gens \cup \left\{ T_1^{u} X_1 - T_0^{u} X_2\right\} $

 \While{$a \neq 0$}
  \State $j++$
 \State $a\gets (u*j) \mod d$
 \State $b\gets (u*j) / d$
 \If{$a > max$}
 \State $gens \gets gens \cup \left\{ T_1^{d-a} X_0^{b+1}X_2^{j-(b+1)}-T_0^{d-a} X_1^j\right\}$
 \State $max\gets a$
 \EndIf
 \If {$a < min$}
 \State $gens\gets gens \cup \left\{ T_1^{a} X_1^j - T_0^{a} X_0^{b}X_2^{j-b}\right\} $
 \State $min\gets a$
 \EndIf
 \EndWhile\\
 \Return $gens$
 \end{algorithmic}
 \end{algorithm}

\begin{example}\label{Ex2}
 Let $I=\langle T_0^{21} ,T_0^{6} T_1^{15}, T_1^{21} \rangle$, so $d=21$ and $u=6$. Algorithm \ref{particular:algorithm} finishes in $\frac{d}{\gcd(d,u)}=7$ steps. First, we add the first two elements to the set of minimal generators of $\Rk(I)$, one for each of the two minimal first syzygies of $I$, corresponding to the minimal generators of $\Syz(I)$, namely:
 $$g_1=T_1^{15} X_0 -T_0^{15} X_1\text{ and } g_2=T_1^{6} X_1-T_0^{6} X_2.$$
 The algorithm proceeds until step $7$ in which $g_6=X_1^7-X_0^2X_2^5$ is added.
 Table \ref{table:example2} shows the trace of the variables across the steps of the algorithm for this example, and the rest of the generators, $g_3, g_4,g_5$. Columns $\max$ and $\min$ only show numbers when they are actually updated, and therefore new elements are added to the set of minimal generators of $\Rk(I)$.

 \renewcommand{\arraystretch}{1.2} 
 \begin{table}
     \begin{tabular}{ccccc|l}
         j&a&b&max&min&gens\\
         \hline
         $1$&$6$&$0$&$6$&$6$&$T_1^{15} X_0-T_0^{15} X_1  ,\, T_1^{6} X_1 - T_0^{6} X_2$ \\
         $2$&$12$&$0$&$12$& &$T_1^{9} X_0X_2-T_0^{9} X_1^2 $\\
         $3$&$18$&$0$&$18$& &$T_1^{3} X_0X_2^2 - T_0^{3} X_1^3 $ \\
         $4$&$3$&$1$& &$3$&$T_1^{3} X_1^4 - T_0^{3} X_0X_2^3 $ \\
         $5$&$9$&$1$&$ $& & \\
         $6$&$15$&$1$&$ $& & \\
         $7$&$0$&$2$&$ $&$0$&$X_1^7 - X_0^2X_2^5$ \\
         \hline
     \end{tabular}
     \caption{Trace of Algorithm \ref{particular:algorithm} for $I=\langle T_0^{21},T_0^6T_1^{15},T_1^{21}\rangle$.}\label{table:example2}
 \end{table}
 \end{example}

\section{Minimal Free Resolution of the Rees ideal}\label{sec:resolution}
In this section we give an explicit description of the minimal free resolution of the defining ideal of the Rees algebra associated to tri-generated monomial ideals in two variables. We thus generalise the results given in \cite{CortadellasDAndrea} for the case of Rees algebras associated to plane curve parametrizations using a different approach.
In Section \ref{sec:graph} we describe what we call the {\em Rees graph} of the ideal $I$, which is constructed using the data obtained in Algorithm \ref{general:algorithm}. We claim that this graph encodes the minimal free resolution of $\Rk(I)$ and that all the information about the minimal resolution can be read off from it (Section \ref{sec:minimal_resolution}). To prove this, we need as preparatory step a Gr\"obner basis of $\Rk(I)$ (Section \ref{sec:grobner}) and from it we obtain a non-minimal free resolution of $\Rk(I)$ (Section \ref{sec:nm_resolution}). Minimalizing this resolution using two basic algebraic reductions yields finally the minimal free resolution (Section \ref{sec:minimal_resolution}).

\subsection{The Rees graph of $I$}\label{sec:graph}
Consider four given integers $d_1,d_2,u_1,u_2$, and the monomial ideal $I=I_{(d_1,d_2,u_1,u_2)}=\langle T_0^{d_1},T_0^{u_1}T_1^{u_2},T_1^{d_2} \rangle \subset \Kbb \left[ T_0,T_1\right]$. Let $G\left(\Rk(I)\right)$
$=\{g_1,\ldots,g_r\}$ be the minimal generating set for the Rees ideal $\Rk(I) \subset \Kbb \left[ T_0,T_1,X_0,X_1,X_2\right]$  sorted as obtained via Algorithm~\ref{general:algorithm} (up to a possible multiplication by $(-1)$ of some elements). Using the data in Algorithm \ref{general:algorithm} we shall construct a graph that encodes the minimal free resolution of $\Rk(I)$. The nodes of the graph correspond to each of the generators in $G\left(\Rk(I)\right)$. To ease the description of the graph, we classify the elements of $G\left(\Rk(I)\right)$ into two types: we say that $g_i\in G\left(\Rk(I)\right)$ is an {\em upper generator} if it was included in step 6, 16 or 25 of Algorithm \ref{general:algorithm}, i.e. if it comes from an element of $\Delta_{max}(d_1,u_1)$, and we say that $g_i$ is a {\em lower generator} if it comes from an element of $\Delta_{min}(d_1,u_1)$, i.e. it was included in step 7, 20 or 29 of Algorithm \ref{general:algorithm}.

We describe  step by step the construction of the graph corresponding to $I$. For this, we will accommodate the nodes of the graph in two rows, a top row formed by nodes corresponding to the {\em upper generators} and a bottom row corresponding to the {\em lower generators}. This display is of course arbitrary, but will help making the description of the graph more convenient. The graph is built in two main steps:
\begin{enumerate}
\item Place the node corresponding to $g_1$ in the top row, and the one corresponding to $g_2$ in the bottom row. 
\item For each of the generators $g_i$, $i\in\{3,\dots,r\}$ proceed in order, doing the following: If $g_i$ is an {\em upper generator} place it as the rightmost element in the top row and draw an edge from the rightmost node in both the top and the bottom row of the graph. Otherwise, if $g_i$ is a {\em lower generator} place it as the rightmost element in the bottom row and draw an edge from the rightmost node in both the top and the bottom row of the graph.
\end{enumerate}

We call this graph the {\em Rees graph of $I$} and denote it by $\mathrm{Graph}(\Rk(I))$ 
or also by $\mathrm{Graph}{(d_1,d_2,u_1,u_2)}$. It has $r$ nodes, $2(r-2)$ edges and $r-3$ triangles. Figure \ref{fig:ex_graph} shows the three steps in the construction of $\mathrm{Graph}(15,13,9,6)$, corresponding to the ideal in Example \ref{Ex1}.

\begin{example}\label{ex:graph}
Let $(d_1,d_2,u_1,u_2)=(15,13,9,6)$ as in Example \ref{Ex1}.

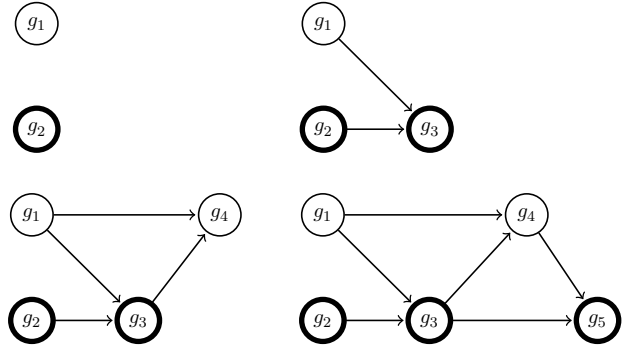
\begin{figure}[ht!]
\hspace{0.05cm}
\minipage{0.30\textwidth}
\hspace{0.5cm}
\begin{tikzpicture}[->,
shorten >=1pt,node distance=2cm,semithick,scale=0.7, transform shape]
  \node[draw, circle,line width= 2pt]  (g2) at (0, 0){$g_2$};
  \node[draw, circle]  (g1) [above of=g2] {$g_1$};
\end{tikzpicture}
\endminipage
\minipage{0.30\textwidth}
\begin{tikzpicture}[->,
shorten >=1pt,node distance=2cm,semithick,scale=0.7, transform shape]

  \node[draw, circle,line width= 2pt]  (g2) at (0, 0){$g_2$};
  \node[draw, circle]  (g1) [above of=g2] {$g_1$};
  \node[draw, circle,line width= 2pt]  (g3) [right of=g2] {$g_3$};

  \path (g1) edge [draw]     (g3)
        (g2) edge          (g3);
\end{tikzpicture}
\endminipage
\vspace{0.5cm}
\minipage{0.30\textwidth}
\begin{center}
\begin{tikzpicture}[->,
shorten >=1pt,node distance=2cm,semithick,scale=0.7, transform shape]

  \node[draw, circle,line width= 2pt]  (g2) at (0, 0){$g_2$};
  \node[draw, circle]  (g1) [above of=g2] {$g_1$};
  \node[draw, circle,line width= 2pt]  (g3) [right of=g2] {$g_3$};
  \node[draw, circle]  (g4)  at (3.53, 2)  {$g_4$};

  \path (g1) edge [draw]     (g3)
        edge           (g4)
        (g2) edge          (g3);
\path (g3) edge [draw]           (g4);
\end{tikzpicture}
\end{center}
\endminipage
\minipage{0.30\textwidth}
\begin{center}
\begin{tikzpicture}[->,
shorten >=1pt,node distance=2cm,semithick,scale=0.7, transform shape]

  \node[draw, circle,line width= 2pt]  (g2) at (0, 0){$g_2$};
  \node[draw, circle]  (g1) [above of=g2] {$g_1$};
  \node[draw, circle,line width= 2pt]  (g3) [right of=g2] {$g_3$};
  \node[draw, circle]  (g4)  at (3.83, 2)  {$g_4$};
  \node[draw, circle,line width= 2pt]  (g5)  at (5.16, 0) {$g_5$};

  \path (g1) edge [draw]           (g3)
        edge           (g4)
        (g2) edge          (g3)
        (g3) edge            (g5)
        edge            (g4)
        (g4) edge       (g5);
  
\end{tikzpicture}
\end{center}
\endminipage
\caption{Step by step of the construction of $\mathrm{Graph}(15,13,9,6)$. The thick nodes correspond to {\em lower generators}.}\label{fig:ex_graph}
\end{figure}
\end{example}

\begin{Notation}\label{defn:TriangleNotation}
We will use the notation $\mathrm{DE}\left( \Gamma \right)$ for the set of directed edges in  a graph $\Gamma=\mathrm{Graph}(d_1,d_2,u_1,u_2)$. As an abuse of notation sometimes we write $(j\longrightarrow k)$ instead of $(g_j \longrightarrow g_k)$ for the edges in $\mathrm{DE}\left( \Gamma \right)$.
In every triangle of $\Gamma$, there is a unique source $j$ and a unique sink $\ell$. Denote the third vertex as $k$; write $(j,k,\ell)$ for the whole triangle. This means in particular that the arrows of $\Gamma$ supported on $g_j,g_k,g_\ell$ are exactly $(j\longrightarrow k)$, $(j\longrightarrow \ell)$, and $(k\longrightarrow \ell)$. Write $\mathrm{Tri}(\Gamma)$ for the set of triangles in $\Gamma$.
\end{Notation}

The main claim of this section (see Theorem \ref{th:minimal_resolution}) is that the minimal free resolution of $\Rk(I)$ is of the form
\[
0 \longrightarrow S^{r-3} \xrightarrow{\hspace{0.2cm} \phi_2 \hspace{0.2cm} } S^{2(r-2)}\xrightarrow{\hspace{0.2cm} \phi_1 \hspace{0.2cm} } S^{r} \xrightarrow{\hspace{0.2cm} \phi_0 \hspace{0.2cm} } \Rk(I) \xrightarrow{\hspace{0.4cm}}  0
\]
and that $\mathrm{Graph}(\Rk(I))$ encodes this resolution in the sense that there is one generator of the first module for each node of the graph, one generator of the second module for each edge in $\mathrm{DE}\left( \mathrm{Graph}(\Rk(I)) \right)$ and one generator of the third module for each triangle in $\mathrm{Tri}\left( \mathrm{Graph}(I) \right)$.  In particular, we have that $\mathrm{pd}(\Rk(I))=2$ if $r>3$ (if $r=3$ then $\mathrm{pd} (\Rk(I))=1$). Also, $\beta_0(\Rk(I))=r$, $\beta_1(\Rk(I))= | \mathrm{DE}\left( \mathrm{Graph}(\Rk(I)) \right)  | = 2(r-2)$ and $\beta_2(\Rk(I))= | \mathrm{Tri}\left( \mathrm{Graph}(I) \right)  | = r-3$ (if $r>3$). Moreover, the differentials in the resolution can be read off from the data in the graph and in Algorithm \ref{general:algorithm}. These claims are proven in the rest of the section, see Theorems \ref{thm:FirstSyzygies} and \ref{thm:SecondSyzygies}.

\subsection{A Gr\"obner basis for $\Rk(I)$}\label{sec:grobner}


Our first step is to build a Gr\"{o}bner basis for $\Rk(I)$, which will be obtained by adding just one element to its minimal generating set. For this, we use a block term order, different from the one used in \cite{CortadellasDAndrea}. The reason for this choice is that this term order will allow us to obtain a Gr\"{o}bner bases of the first and second syzygy modules of $\Rk(I)$ and construct a free resolution. Furthermore, the choice of this term order and the Gr\"obner basis associated to it is convenient for the computation of some homological invariants of $\Rk(I)$ as well as an involutive basis for it, as can be seen in \cite{IOSS24b}.

\begin{definition}
     Let $\sigma=T_0^{a_0}T_1^{a_1}X_0^{b_0}X_1^{b_1}X_2^{b_2}$ and $\tau=T_0^{c_0}T_1^{c_1}X_0^{d_0}X_1^{d_1}X_2^{d_2}$. Then,
    $$\sigma\prec\tau\Leftrightarrow\begin{cases}
        X_0^{b_0}X_1^{b_1}X_2^{b_2}\prec_{\hbox{drl}}X_0^{d_0}X_1^{d_1}X_2^{d_2},\hbox{ or}\\
        X_0^{b_0}X_1^{b_1}X_2^{b_2}=X_0^{d_0}X_1^{d_1}X_2^{d_2}\wedge T_0^{a_0}T_1^{a_1}\prec_{\hbox{drl}}T_0^{c_0}T_1^{c_1}
    \end{cases},$$
    where $\prec_{\hbox{drl}}$ indicates the degree reverse lexicographic ordering with $X_2\prec_{\hbox{drl}}X_0\prec_{\hbox{drl}}X_1$ and $T_0\prec_{\hbox{drl}}T_1$.
\end{definition}
We will be using the same notation from Section \ref{sec:mainAlgorithm}, where the integer $i>1$ is the smallest integer which satisfies $b^{(1)}_i - b^{(1)}_{i-1} = b^{(2)}_i-b^{(2)}_{i-1}$ and $q=\min\{ \frac{d_1}{\gcd(d_1,u_1)}, \frac{d_2}{\gcd(d_2,u_2)}\}$.
\begin{lemma}\label{lem:UniquePureX_1PowerBinomial}
 Let $j$ be a positive integer such that $j<\min\{i,q\}$. Then, for each $k\in\{0,1\}$, there exists exactly one integer $0\leq \ell_{k+1}<d_{k+1}$ such that $\ker(\phi)\subset \Kbb \left[ T_0,T_1,X_0,X_1,X_2\right]$ contains a binomial $\alpha-\beta$ with $\alpha=T_k^{\ell_{k+1}} X_1^j$ and $\gcd(\alpha,\beta)=1$. Moreover, if $k=0$, then $\ell_1=d_1-a^{(1)}_j $, and if $k=1$, then $\ell_2=d_2-a^{(2)}_j.$
 
 Furthermore, if we consider the normalizad minimal generating set $G\left(\Rk(I)\right)$ with respect to the term ordering described above, then all its elements except for the last one, $g_r$, are binomials of the form studied in this lemma. The leading term of the element $g_r=\a_r - \b_r$ 
 is of the form $\a_r= T_0^{\ell_1}T_1^{\ell_2}X_1^j$ with $0\leq \ell_1 < d_1$, $0\leq \ell_2 < d_2$, and $j=\min\{i,q\}$. In fact, only if $i < q$ and $b^{(1)}_i-b^{(1)}_{i-1}=b^{(2)}_i-b^{(2)}_{i-1}=0$ then both $T_0,T_1\mid \a_r$; otherwise, $\alpha_r=T_k^{\ell_{k+1}} X_1^j$ for some $k\in\{0,1\}$ (just like for the rest of the minimal generators).
\end{lemma}
\begin{proof}
 In the same way as we did in Section \ref{sec:mainAlgorithm} (Propositions \ref{prop:generalLimit} and \ref{prop:general3}), for every power $X_1^j$ for $j<\min\{i,q\}$ we consider the two elements that come from a product of powers of $X_0$ and $X_2$ that are lexicographically closest to $\psi(X_1^j)$ in the ordered set $G^{(j)}(I)$, one prior and one after. We must show that the only two binomials described above come from the two relations between the power of $X_1$ and these two elements.\\
 Note that we have $\deg_{T_0}(\phi(X_1^j))=u_1 j$. For any integers $t\geq 0$ and $s\geq 0$, we have $\deg_{T_0}(\phi(X_0^t X_2^s))=d_1 t$. Remember from Section \ref{sec:mainAlgorithm} that $b^{(1)}_j+b^{(2)}_j=j-1$ for $j<i$; hence, $X_0^{b^{(1)}_j+1}X_2^{b^{(2)}_j}$ is the element lexicographically before $X_1^j$, and $X_0^{b^{(1)}_j}X_2^{b^{(2)}_j+1}$ is the element after it. As desired, we get from the first relation $\a_1=T_0^\ell X_1^j$ with $\ell=d_1-a^{(1)}_j$; and from the second $\a_1=T_1^\ell X_1^j$ with $\ell=d_2-a^{(2)}_j$.

 The statements about the elements of $G\left(\Rk(I)\right)$ are clear from Propositions \ref{prop:general1}, \ref{prop:generalLimit}, and \ref{prop:general3}. This can also be seen in Algorithm \ref{general:algorithm}. 
\end{proof}

For this generating set to become a Gr\"obner basis for the order previosly defined, we only add one element representing the trivial syzygy between $\psi(X_0)=T_0^{d_1}$ and $\psi(X_2)=T_1^{d_2}$.

\begin{theorem}\label{prop:degrevlexGB}
  Let $d_1,d_2,u_1,u_2$ be four integers and $G\left(\Rk(I)\right)$ be the normalized minimal generating set w.r.t. the term order described above for the Rees ideal $\Rk(I)\subset \Kbb \left[ T_0,T_1,X_0,X_1,X_2\right]$ of the monomial ideal $I=\langle T_0^{d_1},T_0^{u_1}T_1^{u_2},T_1^{d_2} \rangle \subset \Kbb \left[ T_0,T_1\right]$. Furthermore, let $g_0:= T_1^{d_2}\cdot X_0-T_0^{d_1}\cdot X_2$.
  
Then, $\overline{G}\left(\Rk(I)\right):=\{ g_0 \} \cup G\left(\Rk(I)\right)$ is the minimal Gr\"{o}bner basis of $\Rk(I)$ for that elimination order.
\end{theorem}

\begin{proof}
 For a given monomial $\sigma=T_0^{a_0}T_1^{a_1}X_0^{b_0}X_1^{b_1}X_2^{b_2}$, write $\sigma_X$ for the specialization $\sigma|_{T_0=T_1=1}$. By construction, it is clear that for any binomial $g=\sigma-\tau\in\overline{G}(\Rk(I))$, we have $\sigma_X\neq\tau_X$, and thus the leading term of $g$ with respect to our elimination order $\prec$ only depends on $\sigma_X$ and $\tau_X$. Hence, $\lt(g_0)=T_1^{d_2}X_0$, and by Lemma \ref{lem:UniquePureX_1PowerBinomial}, $\lt(g_r)=T_0^{\ell_1}T_1^{\ell_2} X_1^j$, for some integers $\ell_1,\ell_2$ and $j$ with $0\leq \ell_1 < d_1$, $0\leq \ell_2 < d_2$, and $j=\min\{i,q\}$; and for any other $g\in \overline{G}(\Rk(I))$, we have $\lt(g)$ is either $T_0^{\ell_1} X_1^j$ or $T_1^{\ell_2} X_1^j$  for some  $1\leq j < \min\{i,q\}$ as seen in Lemma \ref{lem:UniquePureX_1PowerBinomial}. 
 
 Let now $\sigma=T_i^\ell X_1^m\in\lt(G(\Rk(I)))$ where $i\in\{0,1\}$. Then $\lt(\Rk(I))$ contains no strict divisor of $\sigma$. Indeed, as the exponent $\ell$ of $T_i$ is derived in Algorithm~\ref{general:algorithm} from an extreme value, and as such a divisor $\alpha$ would be of the form stated in Lemma~\ref{lem:UniquePureX_1PowerBinomial}, the existence of a strict divisor would contradict the correctness of the algorithm (Theorem \ref{thm:general}).
 
 Now consider any binomial $b=\alpha-\beta\in\Rk(I)$, where $\alpha\succ\beta$. We need to show that there is an element $g\in\overline{G}(I)$ such that $\lt(g)$ divides $\lt(b)=\alpha$. If the monomials $\alpha$ and $\beta$ are not coprime, then we have that $b/\gcd(\alpha,\beta)$ is also in $\ker(\phi)$. So from now on we consider that $\alpha$ and $\beta$ are coprime. It is clear that $\deg(\alpha_X)=\deg(\beta_X)>0$. We now analyse several cases:
 
Case 1: $\alpha_X=X_1^s$ for some positive integer $s$. If $s\geq\min\{i,q\}$, then $\a$ is either divisible by $\lt(g_r)$ as we saw from Propositions \ref{prop:general1} and \ref{prop:generalLimit}; or otherwise we are not looking at a binomial with the smallest coefficients $T_i$ and then $\lt(g_1)$ or $\lt(g_2)$ must divide $\a$. If $s<\min\{i,q\}$, then write $\alpha=T_i^r X_1^s$ for some integer $r$ and for some $i\in\{0,1\}$. If $r\geq d$, it is obvious that $\alpha$ is divisible by either $\lt(g_1)$ or $\lt(g_2)$. If $r<d$, then $\alpha-\beta$ is a binomial of the form stated in Lemma~\ref{lem:UniquePureX_1PowerBinomial}. Moreover, there is a maximal integer $1\leq j\leq s$ such that $\lt(G(I))$ contains a term of the form $\sigma=T_k^{\ell_{k+1}}X_1^{j}$, where $\ell_{k+1}$ is chosen via an extreme property in Algorithm~\ref{general:algorithm}. It follows that $\sigma$ divides $\alpha$.

Case 2: $\alpha_X$ is divisible by $X_0$ and $\beta_X$ is divisible by $X_2$. Then $\deg_{T_1}(\phi(\alpha_X))\leq \deg_{T_0}(\phi(\beta_X))-d_2$. But then, since $\phi(\alpha)=\phi(\beta)$, $T_1^{d_2}|\alpha$. Hence, $T_1^{d_2} X_0=\lt(g_0)|\alpha$.

Case 3: $\alpha_X$ is divisible by $X_2$ and $\beta_X$ is not divisible by $X_0$, This case cannot occur, because then $\beta_X$ is a pure power of $X_1$ and then $\beta$ would be the leading term of the binomial $\alpha-\beta$.

By Lemma \ref{lem:UniquePureX_1PowerBinomial} it is clear that $X_1 \mid \lt(g_i)$ and $X_0,X_2 \nmid \lt(g_i)$ for all $g_i \in G\left(\Rk(I)\right)$ and it is also clear (from Algorithm \ref{general:algorithm}) that there is no divisibility among the leading terms of the minimal generators. Besides, note that $X_1 \nmid \lt(g_0)$ but $X_0\mid \lt(g_0)$, hence we can guarantee that there is no divisibility either among the leading terms of $\overline{G}\left(\Rk(I)\right)$. Moreover, since this set is already a normalized  Gr\"{o}bner basis for that specific term order with the smallest size possible we deduce its minimality.
\end{proof}

\subsection{A free resolution of $\Rk(I)$}\label{sec:nm_resolution}
The goal of this section is to set the ground for the proof that $\mathrm{Graph}(\Rk(I))$ encodes the minimal free resolution of $\Rk(I)$. For this, we first augment the graph with the new generator $g_0$ introduced in the previous paragraphs. As we will see below, the element $g_0$ behaves in a similar way as all other lower generators so we place it as the leftmost element of the bottom row of the graph. The new augmented graph $\overline{\mathrm{Graph}}(I)=\overline{\mathrm{Graph}}(d_1,d_2,u_1,u_2)$ has one more vertex, corresponding to $g_0$, two new edges, namely $g_0\rightarrow g_2$ and $g_1\rightarrow g_2$; and a new triangle, formed by the vertices $g_1, g_2, g_3$ and the edges that connect them. The augmented graph $\overline{\mathrm{Graph}}(I)$ encodes a free resolution of $\Rk(I)$. In Section \ref{sec:minimal_resolution} this resolution will be reduced to a minimal one.

\begin{example}\label{ex:AssociatedGraph}
Figure \ref{fig:augmented_example} shows the augmented graph $\overline{Graph}(15,13,9,6)$.
     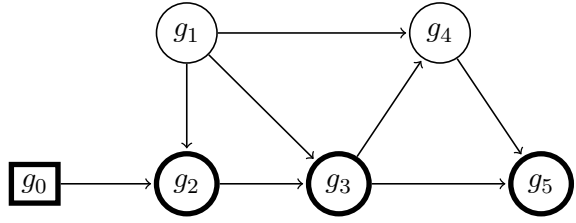
\begin{figure}[ht!]
\begin{center}
\begin{tikzpicture}[->,
shorten >=1pt,node distance=2cm,semithick]

  \node[draw, rectangle, line width= 2pt]  (g0)                 {$g_0$};
  \node[draw, circle, line width= 2pt]  (g2) [ right of=g0] {$g_2$};
  \node[draw, circle]  (g1) [above of=g2] {$g_1$};
  \node[draw, circle, line width= 2pt]  (g3) [right of=g2] {$g_3$};
  \node[draw, circle]  (g4)  at (5.33, 2)  {$g_4$};
  \node[draw, circle, line width= 2pt]  (g5)  at (6.66, 0) {$g_5$};

  \path (g0) edge[draw]      (g2)
        (g1) edge[draw]            (g2)
        edge           (g3)
        edge           (g4)
        (g2) edge          (g3)
        (g3) edge            (g5)
        edge            (g4)
        (g4) edge       (g5);
\end{tikzpicture}
\end{center}
\caption{$\overline{\mathrm{Graph}}(15,13,9,6)$.}\label{fig:augmented_example}
\end{figure}
\end{example}

For the analysis of the syzygies of the Gr\"{o}bner basis $\overline{G}\left(\Rk(I)\right)$, recall the two sets $\Delta_{\max/\min}(d_k,u_k)$ defined in Section \ref{sec:mainAlgorithm}. Since nothing relevant happens for steps after $t=\min\{i,q\}$, we can remove all values strictly higher than $t$ from every $\Delta_{\max/\min}(d_k,u_k)$. Also, let $\Delta(d_1,d_2,u_1,u_2):=\{\deg_{X_1}(\lt(g))\mid g\in  G\left(\Rk(I)\right) \}$. Note that $\Delta(d_1,d_2,u_1,u_2)=\Delta_{\max}(d_1,u_1)\cup \Delta_{\min}(d_1,u_1)=\Delta_{\min}(d_2,u_2)\cup \Delta_{\max}(d_2,u_2)$.

We also need to make use of the following monomial ideals:

\begin{definition}\label{defn:AssociatedColonIdeals}
Write $\overline{G} \left(\Rk(I)\right) =\{g_0,g_1,\ldots,g_r\}$. For $0\leq j<r$ consider the quotient monomial ideals $$M_j=\langle \lt(g_{j+1}),\ldots,\lt(g_r) \rangle:\langle\lt(g_j)\rangle,$$
and let $G(M_j)$ be the minimal monomial generating set of $M_j$.
\end{definition}

By a well-known construction, the set of syzygies of $\overline{G} \left(\Rk(I)\right)$ induced by multiplying any $g_j\in \overline{G} \left(\Rk(I)\right)$ by any minimal generator $m\in G(M_j)$ and reducing to zero with respect to $\overline{G} \left(\Rk(I)\right)$ is a Gr\"{o}bner basis of $\Syz\left(\overline{G}\left(\Rk(I)\right)\right)$ with respect to a suitable module term order (e.g. \cite[Cor. 1.11]{BerkeschSchreyer:SyzygiesFiniteLengthModulesRandomCurves}, \cite{LaScalaStillman:StrategiesForComputingMinFreeRes}). We will use this fact. 


In order to understand the structure of $\overline{\mathrm{Graph}}(I)$, we will use the colon ideals $M_j$. For their description, we define an auxiliary map.

\begin{definition}\label{defn:AuxiliaryMapToGBIndexSet}
 For each $\delta\in\Delta(d_1,d_2,u_1,u_2)\setminus\{1\}$, there is a unique $g_j\in\overline{G}\left(\Rk(I)\right)\setminus\{g_0,g_1,g_2\}$ such that $\deg_{X_1}(g_j)=\delta$. Write $\iota(\delta):=j$.
\end{definition}

Note that $\iota:\Delta(d_1,d_2,u_1,u_2)\setminus\{1\}\longrightarrow \{3,\ldots,r\}$ is bijective.

\begin{proposition}\label{prop:ColonIdealStructure}
Let $t=\min\{i,q\}$. $\overline{G}\left(\Rk(I)\right)$, $M_j$, and $G(M_j)$ be given as in Definition~\ref{defn:AssociatedColonIdeals}. Then:
\begin{itemize}
 \item $G(M_0)=\{X_1\}$,
 \item $G(M_1)=\{T_1^{d_2-u_2}\}\cup \{\lt(g_{\iota(\zeta_1)})/X_1,\ldots,\lt(g_{\iota(\zeta_\ell)})/X_1,\lt(g_{\iota(\epsilon)})/T_0^{\deg_{T_0}\left( \lt(g_{\iota(\epsilon)})\right)}X_1\}$, where:
 $\epsilon=\min\{\gamma\in\Delta_{\max}(d_1,u_1)\cup \{t\}\mid \gamma>1\}$, and $\zeta_1<\cdots<\zeta_\ell$ are the elements of $\Delta_{\min}(d_1,u_1)$ between $1$ and $\epsilon$,
 \item $G(M_2)=\{\lt(g_{\iota(\zeta_1)})/X_1,\ldots,\lt(g_{\iota(\zeta_\ell)})/X_1,\lt(g_{\iota(\epsilon)})/T_1^{\deg_{T_1}\left( \lt (g_{\iota(\epsilon)})\right)}X_1\}$, where:
 $\epsilon=\min\{\gamma\in\Delta_{\min}(d_1,u_1)\cup \{t\} \mid \gamma>1\}$, and $\zeta_1<\cdots<\zeta_\ell$ are the elements of $\Delta_{\max}(d_1,u_1)$ between $1$ and $\epsilon$,
 \item For $2<j<t$, \begin{itemize}
                   \item If $\delta:=\iota^{-1}(j) \in\Delta_{\min}(d_1,u_1)$, then $$G(M_j)=\{ \lt(g_{\iota(\zeta_1)})/X_1^\delta,\ldots,\lt(g_{\iota(\zeta_\ell)})/X_1^\delta,\lt(g_{\iota(\epsilon)})/T_1^{\deg_{T_1}\left( \lt(g_{\iota(\epsilon)})\right)}X_1^\delta \},$$ where:
                   $\epsilon=\min\{\gamma\in\Delta_{\min}(d_1,u_1)\cup \{t\}\mid \gamma>\delta\}$, and $\zeta_1<\cdots<\zeta_\ell$ are the elements of $\Delta_{\max}(d_1,u_1)$ between $\delta$ and $\epsilon$,
                   \item If $\delta:=\iota^{-1}(j) \in\Delta_{\max}(d_1,u_1)$, then $$G(M_j)=\{ \lt(g_{\iota(\zeta_1)})/X_1^\delta,\ldots,\lt(g_{\iota(\zeta_\ell)})/X_1^\delta,\lt(g_{\iota(\epsilon)})/T_0^{\deg_{T_0}\left(\lt(g_{\iota(\epsilon)})\right)}X_1^\delta \},$$ where:
                   $\epsilon=\min\{\gamma\in\Delta_{\max}(d_1,u_1)\cup \{t\}\mid \gamma>\delta\}$, and $\zeta_1<\cdots<\zeta_\ell$ are the elements of $\Delta_{\min}(d_1,u_1)$ between $\delta$ and $\epsilon$.
                  \end{itemize}

\end{itemize}

\end{proposition}
\begin{proof}
Recall that $\overline{G}\left(\Rk(I)\right)=\{g_0,g_1,g_2,\ldots,g_r\}$, where $g_0=T_1^{d_2} X_0-T_0^{d_1} X_2$, $g_1=T_0^{d_1-u_1}X_1-T_1^{u_2}X_0$, and $g_2=T_1^{d_2-u_2}X_1-T_0^{u_1}X_2$. Moreover, the leading terms of the other elements $g_j$, where $2<j<r$ are of the form $\lt(g_j)=P_j\cdot X_1^{q_j}$ with $2\leq q_j<t$ and $P_j$ a pure power of $T_0$ or of $T_1$. Moreover, by Algorithm~\ref{general:algorithm}, $P_j$ is a power of $T_0$ with positive exponent if and only if $j<r$ and $\iota^{-1}(j)\in\Delta_{\max}(d_1,u_1)\cap \Delta_{\min}(d_2,u_2)$; it is a power of $T_1$ with positive exponent if and only if $j<r$ and $\iota^{-1}(j)\in\Delta_{\min}(d_1,u_1)\cap \Delta_{\max}(d_2,u_2)$. It is also immediate from Algorithm~\ref{general:algorithm} that the exponents of $T_0$ respectively $T_1$ form decreasing sequences for increasing indices $j$. Furthermore, $\lt(g_r)$ is either a pure power of $X_1$ or a power of $X_1$ with coefficients $T_0$ and $T_1$ of lowest degree.

From what has been stated, it is clear that $G(M_0)=\{X_1\}$. For any $j$ with $0<j<r$, the $X_1$-degrees of $\lcm(\lt(g_j),\lt(g_{j+1}))/\lt(g_j),\ldots,\lcm(\lt(g_j),\lt(g_{r}))/\lt(g_j)$ form an increasing sequence. Let $\lt(g_j)=T_k^{p_j}X_1^{q_j}$, where $k\in\{0,1\}$ (i.e. if $g_j$ is an upper generator $k=0$, and if it is a lower generator $k=1$). It is clear that, for the first index $\ell>j$ for which $\lt(g_\ell)$ is not divisible by $T_{1-k}$, we have that $\lcm(\lt(g_j),\lt(g_{\ell})/\lt(g_j)$ is a pure power of $X_1$. This implies $$G(M_j)=\{\lcm(\lt(g_j),\lt(g_{j+1}))/\lt(g_j),\ldots,\lcm(\lt(g_j),\lt(g_{\ell}))/\lt(g_j)\}.$$
Moreover, $\lt(g_{j+1}),\ldots,\lt(g_{\ell-1})$ are all divisible by $T_{1-k}$ and thus, they are all of the same type of generators, while both $\lt(g_j)$ and $\lt(g_\ell)$ are of the opposite type (lower generators vs. upper generators).

In the case that such a index $l$ cannot be found it means that $l=t$ with $\iota(t)=r$, and since $\lt(g_r)$ is either a pure power $X_1^t$ or the highest power of $X_1$ with lowest $T_k$ coefficient, then either way the same argument that we claimed for the other case still holds. The claimed statements are now obvious.
\end{proof}

\begin{proposition}\label{prop:colon_edges}
Let the monomial ideals $M_j$ and their minimal generating sets $G(M_j)$ be given as in Definition~\ref{defn:AssociatedColonIdeals} for the Gr\"{o}bner basis $\overline{G}\left(\Rk(I)\right)$. There exists a directed edge $(j\longrightarrow \ell) \in \mathrm{DE}(\overline{\mathrm{Graph}}(I))$ for every pair of indices $(j,\ell)$ with $j<\ell$ such that $\lcm(\lt(g_j),\lt(g_\ell))/\lt(g_j)\in G(M_j)$. 
\end{proposition}

\begin{proof}
The result is obtained immediately from the description of the ideals $M_j$ that was just given in Proposition \ref{prop:ColonIdealStructure}.
\end{proof}

\begin{corollary}\label{cor:AcyclicityArrowAndTriangleCount}
 For $\overline{G}(I)=\{g_0,g_1,\ldots,g_r\}$, $\overline{\mathrm{Graph}}(I)$ is acyclic and contains $2r-2$ directed edges and $r-2$ triangles.
\end{corollary}
\begin{proof}
 $\overline{\mathrm{Graph}}(I)$ is acyclic because for any directed edge $(j\longrightarrow k)$ in it, we must have, by construction, $j<k$.
 
 There are two directed edges pointing to $g_2$ in the graph, these correspond to the $X_1\in G(M_0)$ and $T_1^{d_2-u_2}\in G(M_1)$. All other edges have targets $g_k$ for some $k>2$. Moreover, from each $g_j$  with $j\geq 1$ of a given type ($\Delta_{\min}(d_1,u_1)$ vs. $\Delta_{\max}(d_1,u_1)$), there is an arrow to the next $g_k$ of the same type and to all $g_\ell$ between $g_j$ and $g_k$ (of the the opposite type). Thus, for any $g_k$ with $k\geq2$ there are exactly two edges with target $g_k$: One from the previous $g_j$ of the same type as $g_k$, and one from the $g_\ell$ of the opposite type where $\ell$ is the maximal possible index. Thus, there are $2(r-1)=2r-2$ edges.
 
 Now consider any triangle in the graph. By acyclicity there is a unique sink $g_k$ in the triangle. As $g_2$ is targeted by two edges coming from $g_0$ and $g_1$, but no edge from $g_0$ to $g_1$ can ever exist, we have $k>2$. Without loss of generality, assume $\deg_{X_1}(g_k)\in\Delta_{\max}(d_1,u_1)$ i.e. an upper generator. As we saw $g_k$ is targeted by exactly two edges, one coming from the previous upper generator $g_j$, and the other from the previous lower 
 generator i.e. corresponding to the minimum with largest subindex $\ell<k$, $g_\ell$. It is obvious that there is at most one triangle with sink $g_k$. So now, we show that there does exist a triangle with this sink: If $g_\ell$ lies between $g_j$ and $g_k$ i.e. $j<\ell<k$, then there is a directed edge $(j\longrightarrow \ell)$, completing the triangle; otherwise, $\ell<j$. But then, as $g_\ell$ targets the next minimum $g_h$ with $h>k$, and all other maximums in between, so it targets $g_j$ and $g_k$ as desired. Either way, we obtain a triangle.
\end{proof}



\begin{Notation}
Observe that we can write every  binomial $g_j\in \overline{G}(\Rk(I))$ as $g_j=\lt(g_j)-\mathrm{tt}(g_j)$, where $\mathrm{tt}(g_j)$ stands for the tail term of the binomial $g_j$. To ease the description of the formulas of the next two theorems let us now introduce the following notation:
\begin{itemize}
    \item[-]  $v(a,b)=\lcm(\lt(g_a),\lt(g_b))/\lt(g_a)$, 
    \item[-]  $w(a,b)=\gcd(\mathrm{tt}(g_a),\mathrm{tt}(g_b))$.
\end{itemize}
\end{Notation}
 Using this notation we can state and prove the first main result of the section.

\begin{theorem}\label{thm:FirstSyzygies}
 Write $\overline{G}\left(\Rk(I)\right)=\{g_0,g_1,\ldots,g_r\}$, and denote by $\{\mathbf{e}_0,\mathbf{e}_1,\ldots,\mathbf{e}_r\}$ the canonical basis of $(\mathbb{K}[T_0,T_1,X_0,X_1,X_2])^{r+1}$. A Gr\"{o}bner basis $S^{(1)}$ of $\Syz\left(\overline{G}\left(\Rk(I)\right)\right)$ is given by:
\[
  S^{(1)}=  \{ \mathbf{s}_{(j,k)}^{(1)}=v(j,k)\mathbf{e}_j-v(k,j)\mathbf{e}_k\\+w(j,k)\mathbf{e}_h\mid \{ (j\rightarrow k),(h\rightarrow k)\} \subset \mathrm{DE}\left( \overline{\mathrm{Graph}}(I) \right)
  \},
 \]
Observe that $|S^{(1)}|=|\mathrm{DE}\left(\overline{\mathrm{Graph}}(I)\right)|=2r-2$.
\end{theorem}
\begin{proof}
 Recall that $\overline{G}\left(\Rk(I)\right)=\{g_0,g_1,g_2,\ldots,g_r\}$, where $g_0=T_1^{d_2} X_0-T_0^{d_1} X_2$, $g_1=T_0^{d_1-u_1}X_1-T_1^{u_2}X_0$, and $g_2=T_1^{d_2-u_2}X_1-T_0^{u_1}X_2$. Thus, the two nodes that target $g_{k=2}$ are $g_0$ and $g_1$, which gives us as a result the first two elements of $S^{(1)}$: $s^{(1)}_{(0,2)} = X_1\mathbf{e}_0 - T^{u_2}_1 \mathbf{e}_2 +T^{u_1}_0 X_2\mathbf{e}_1$, and $s^{(1)}_{(1,2)} =T^{d_2-u_2}_1\mathbf{e}_1 - T^{d_1-u_1}_0 \mathbf{e}_2 + 1\mathbf{e}_0$. One verifies by direct computation that these two elements of $S^{(1)}$ are in $\Syz\left(\overline{G}\left(\Rk(I)\right)\right)$ and arise from reductions to zero with respect to $\overline{G}\left(\Rk(I)\right)$ of $X_1g_0$ and $T_1^{d_2-u_2}g_1$, respectively.
 
 It remains to be shown that the other elements of $S^{(1)}$ are in bijection to the remaining elements of $\cup_j G(M_j)$ as listed in Proposition~\ref{prop:ColonIdealStructure} and arise from reductions to zero of $mg_j$ with respect to $\overline{G}\left(\Rk(I)\right)$, where $m\in G(M_j)$. A bijection is clearly given via the arrow count in $\mathrm{DE}\left(\overline{\mathrm{Graph}}(I)\right)$, using the arrow $(j\longrightarrow k)$. Consider the element $g_j$. Clearly, $m:=v(j,k)\in G(M_j)$. It is also obvious that the first step of the reduction of $mg_j$ is subtraction of $v(k,j)g_k=:\tilde{m}g_k$. There remains the S-polynomial $s:=-m\cdot\mathrm{tt}(g_j)+\tilde{m}\cdot\mathrm{tt}(g_k)$. Since $j<k$, $\deg_{X_1}(\lt(g_j))<\deg_{X_1}(\lt(g_k))$; thus, $\deg_{X_1}\left(m\cdot\mathrm{tt}(g_j)\right)=\deg_{X_1}(m)>0=\deg_{X_1}\left(\tilde{m}\cdot\mathrm{tt}(g_k)\right)$. This implies $\lt(s)=m\cdot\mathrm{tt}(g_j)$. Without loss of generality, assume $\deg_{X_1}(\lt(g_j))\in\Delta_{\min}(d_1,u_1)$. Then by Lemma~\ref{lem:MinMaxGeneralRelation} item 1a or 3a, $\deg_{X_1}(\lt(s))\in\Delta_{\max}(d_1,u_1)$ and by item 1b or 3b of the same lemma it is the highest such degree below $\deg_{X_1}(\lt(g_k))$. Thus $g_h$ with $h:=\iota\left(\deg_{X_1}(\lt(s))\right)$ is the second node in $\overline{\mathrm{Graph}}(I)$ targeting $g_k$. Finally, the binomial $\tilde{s}:=s/w(j,k)$ is made up out of two coprime monomials and $\deg_{X_1}(\lt(\tilde{s}))=\deg_{X_1}(\lt(s))$. Thus by Lemma~\ref{lem:UniquePureX_1PowerBinomial}, $\tilde{s}=-g_h$; hence, $s+w(j,k)g_h=0$, finishing the reduction, as claimed.
\end{proof}

A Gr\"obner basis for the second syzygy module $\Syz^2\left(\overline{G}\left(\Rk(I)\right)\right)$ can be also obtained from $\overline{\mathrm{Graph}}(I)$. In this case, the triangles in the graph play a fundamental role.

\begin{theorem}\label{thm:SecondSyzygies}
 Write $\overline{G}=\{g_0,g_1,\ldots,g_r\}$, and denote by $\{\mathbf{f}_{(j,k)}\mid (j\longrightarrow k)\in\overline{\mathrm{Graph}}(I)\}$ the canonical basis of $(\mathbb{K}[T_0,T_1,X_0,X_1,X_2])^{2r-2}$. A Gr\"{o}bner basis $S^{(2)}$ of $\Syz^2\left(\overline{G}\left(\Rk(I)\right) \right)$ is given by:
 \begin{align*}
  S^{(2)}=
 \{  & \mathbf{s}_{(j,k,\ell)}^{(2)}=v(h,k)\mathbf{f}_{(j,k)}-v(l,k)\mathbf{f}_{(j,\ell)}+v(l,j)\mathbf{f}_{(k,\ell)}-w(j,k)\mathbf{f}_{(h,k)}\mid\\& (j,k,\ell)\in\mathrm{Tri}\left(\overline{\mathrm{Graph}}(I)\right)\wedge (h\longrightarrow k)\in\mathrm{DE}\left( \overline{\mathrm{Graph}}(I)\right) \wedge h \neq j
 \},
 \end{align*}

Moreover, $\langle S^{(2)}\rangle\subset (\mathbb{K}[T_0,T_1,X_0,X_1,X_2])^{2r-2}$ is a free submodule. Observe that $|S^{(2)}|=|\mathrm{Tri}\left(\overline{\mathrm{Graph}}(I)\right)|=r-2$.
.
\end{theorem}
\begin{proof}
 We first look at the leading term set of $S^{(1)}$. It is in bijection to $\cup_j G(M_j)$, where the monomial ideals $M_j$ are as in Proposition~\ref{prop:ColonIdealStructure}. We order $S^{(1)}$ linearly in the way induced by $G(M_0)\cup G(M_1)\cup\ldots \cup G(M_{r-1})$ taken in the order of writing; the sets $G(M_j)$ themselves are ordered as written in Proposition~\ref{prop:ColonIdealStructure}. We now take the induced colon ideals. Their analysis can be split into that of the colon ideals of $G(M_j)$, which are associated to the module component supported on $\mathbf{e}_j$ ($0\leq j\leq r$). But, as noted in the proof of Proposition~\ref{prop:ColonIdealStructure}, $G(M_j)=\{T_i^{a_1}X_1^{b_1},\ldots,T_i^{a_s}X_1^{b_s}\}$ ($i\in\{0,1\}$) for a decreasing sequence of non-negative integers $a_1>a_2>\cdots >a_s\geq 0$ and an ascending sequence of non-negative integers $0\leq b_1<\cdots <b_s<c$. (These integers depend on $j$.) Thus, all colon ideals are generated by a single pure $X_1$-power. (Note that the singleton $G(M_0)$ does not induce any colon ideal.) The freeness of $\Syz^2\left(\overline{G}(I)\right)$ follows. 
 
 Now, for a given $j$, let $p=X_1^{b_{s+1}-b_s}= \lcm(T_i^{a_{s+1}}X_1^{b_{s+1}},T_i^{a_s}X_1^{b_s})/T_i^{a_s}X_1^{b_s}$ be the single generator of one of the colon ideals induced by $G(M_j)$. There are indexes $k<\ell$ (in fact, $\ell=k+1$), such that $T_i^{a_{s}}X_1^{b_{s}}=v(j,k)$ and $T_i^{a_{s+1}}X_1^{b_{s+1}}=v(j,\ell)$. The module term $p\mathbf{f}_{j,k}$ is the leading term of one of the generators of $\Syz\left(S^{(1)}\right)$ and every leading term arises in this way. Thus the generating set of $\Syz\left(S^{(1)}\right)$ induced by the colon ideal structure of the leading term set of $S^{(1)}$ is in bijection to $\mathrm{Tri}\left(\overline{\mathrm{Graph}}(I)\right)$ via $p\mathbf{f}_{j,k}\mapsto (j,k,\ell)$.
 

 
 Now consider $\mathbf{s}_{(j,k,\ell)}^{(2)}$ for a given triangle $(j,k,\ell)$ with $\ell\geq3$. As $g_k$ does not correspond to the last element in the enumeration of $G(M_j)$, $\deg_{X_1}(g_j)$ and $\deg_{X_1}(g_k)$ are of opposite types (upper vs. lower generators) by Corollary~\ref{cor:AcyclicityArrowAndTriangleCount}.  Let $g_h$ be the second node in $\overline{\mathrm{Graph}}(I)$ targeting $g_k$. Necessarily, $g_h$ is of the same tpye as $g_k$. Thus, by Lemma~\ref{lem:MinMaxGeneralRelation}, we have that $\deg \big( v(h,k) \big) =\deg_{X_1} \big( v(h,k) \big) =\deg_{X_1}(g_k)-\deg_{X_1}(g_h)$ is in the same type as $\deg_{X_1}(g_j)$ and it is the maximal such degree below $\deg_{X_1}(g_k)$. But $g_j$ targets $g_k$. Therefore, $\deg_{X_1}(g_j)=\deg_{X_1}(g_k)-\deg_{X_1}(g_h)$, this can also be deduced from Corollary \ref{cor:nextrelevant} since $\deg_{X_1}(g_j)+\deg_{X_1}(g_h)=\deg_{X_1}(g_k)$. A similar argument shows that $\deg(p)=\deg_{X_1}(g_\ell)-\deg_{X_1}(g_k)=\deg_{X_1}(g_j)$. Thus, $p=v(h,k)$ as claimed and the leading module term of
 $\mathbf{s}_{(j,k,\ell)}^{(2)}$ is 
 $v(h,k)\mathbf{f}_{(j,k)}$. 
 Substituting $\mathbf{s}_{(j,k)}^{(1)}$, we obtain an element $\mathbf{s}'\in\Syz\left(\overline{G}\right)$ with leading term
 $v(h,k)v(j,k)\mathbf{e}_{j}
 =\big(X_1^{\deg_{X_1}(g_j)}\big)
 \big(\lt(g_k)/X_1^{\deg_{X_1}(g_j)}\big)
 \mathbf{e}_j=\lt(g_k)\mathbf{e}_j$. 
 We can reduce $\mathbf{s}'$ using $\mathbf{s}_{(j,\ell)}^{(1)}$, which has the leading term $v(j,\ell)\mathbf{e}_j$. 
 Again using Proposition~\ref{prop:ColonIdealStructure} and Lemma~\ref{lem:MinMaxGeneralRelation}, we see that $\deg_{X_1}\big(v(j,\ell)\big)=\deg_{X_1}(g_k)$. 
 It is possible that $v(j,\ell)$ is additionally divisible by a $T_i$ ($i\in\{0,1\}$);
 but in any case, $v(j,\ell)v(\ell,k)=\lt(g_k)$. 
 Thus, we reduce $\mathbf{s}'$ by subtracting $v(\ell,k)\mathbf{s}_{(j,\ell)}^{(1)}$, obtaining $\mathbf{s}''$.
 In its support, we find $v(h,k)w(j,k)\mathbf{e}_h$.
 The syzygy $\mathbf{s}_{(h,k)}^{(1)}\in S^{(1)}$ has leading term $v(h,k)\mathbf{e}_h$. 
 Hence, we reduce $\mathbf{s}''$ by subtracting $w(j,k)\mathbf{s}_{(h,k)}^{(1)}$, obtaining $\mathbf{s}'''$. There is only one module monomial in its support which is divisible by $X_1$, namely $-v(h,k)v(k,j)\mathbf{e}_k$. Using Lemma~\ref{lem:MinMaxGeneralRelation}, we can simplify this to $-\lt(g_j)\mathbf{e}_k$. We notice that the element $\mathbf{s}_{(k,l)}^{(1)}\in S^{(1)}$ has leading term $v(k,\ell)\mathbf{e}_k$. By Lemma~\ref{lem:MinMaxGeneralRelation}, $\deg_{X_1}\big( v(k,\ell) \big)=\deg_{X_1}(g_j)$. It is possible that $ v(k,\ell)$ is additionally divisible by a $T_i$ ($i\in\{0,1\}$); but in any case, $  v(k,\ell) v(\ell,j)=\lt(g_j)$. Hence, we reduce $\mathbf{s}'''$ by adding $v(\ell,j)\mathbf{s}_{(k,\ell)}^{(1)}$, obtaining $\mathbf{s}''''$. One can check that no term divisible by $X_1$ remains in $\mathbf{s}''''$. Assume $\mathbf{s}''''\neq\mathbf{0}$; then its leading term would not be divisible by $X_1$, in contradiction to Theorem~\ref{thm:FirstSyzygies}. Hence, $\mathbf{s}''''=\mathbf{0}$.

 Note that even though $g_{h=0}$ is technically of neither type, but we have that $G(M_0)=X_1$ and $(0\rightarrow2)  \in \mathrm{DE}\left(\overline{\mathrm{Graph}}(I)\right)$. Hence, $g_0$ behaves exactly like a lower generator of the type coming from $\Delta_{\min}(d_1,u_1)$. Moreover, it can be verified by direct computation that the claims stated before also hold for the particular triangle $(1,2,3)\in\mathrm{Tri}\left(\overline{\mathrm{Graph}}(I)\right)$ with $h=0$, more about this particular triangle will follow in the next subsection.
\end{proof}

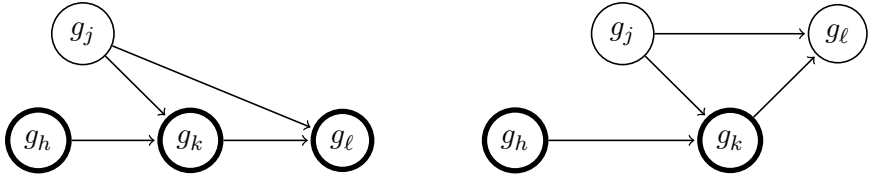
\begin{figure}[ht!]
\begin{center}
\begin{tikzpicture}[->,
shorten >=1pt,node distance=2cm,semithick]

  \node[draw, circle]  (gj)  {$g_j$};
  \node[draw, circle, line width= 2pt]  (gk) [below right of=gj] {$g_k$};
  \node[draw, circle, line width= 2pt]  (gh) [left of=gk] {$g_h$};
  \node[draw, circle, line width= 2pt]  (gl) [right of=gk]  {$g_\ell$};

  \path (gj)
        edge           (gk)
        edge             (gl)
        (gh) edge          (gk)
        (gk) edge       (gl);
\end{tikzpicture}\;\;\;\;\;\;\;\;\;\;\;\;
\begin{tikzpicture}[->,
shorten >=1pt,node distance=2cm,semithick]

  \node[draw, circle]  (gj)  {$g_j$};
  \node[draw, circle,line width= 2pt]  (gh) [below left of=gj] {$g_h$};
  \node[draw, circle, line width= 2pt]  (gk) [below right of=gj] {$g_k$};
  \node[draw, circle]  (gl) [above right of=gk]  {$g_\ell$};

  \path (gj)
        edge           (gk)
        edge             (gl)
        (gh) edge          (gk)
        (gk) edge       (gl);
\end{tikzpicture}
\end{center}
\caption{The nodes $g_h$, $g_j$, $g_k$, $g_\ell$ appearing in the the construction of the syzygy $\mathbf{s}_{(j,k,\ell)}^{(2)}$ in the proof of Theorem ~\ref{thm:SecondSyzygies} are situated as illustrated here. There are two cases. Note that $g_j$, $g_k$ must belong to opposite types (upper vs. lower generators); the only difference between the two cases is whether $g_\ell$ is of the same type as $g_k$ (seen on the left) or of the same the type as $g_j$ (seen on the right). Without loss of generality we assumed both $g_h$ and $g_k$ are lower generators, there exists the other possibility of both being upper generators in which case the images would be flipped upside down.}\label{fig:triangules}
\end{figure}

\begin{figure}[ht!]
\begin{center}
\begin{tikzpicture}[
  delim yshift=1ex,
  delim xshift=.1em,]
\matrix[
  matrix of math nodes,
  row sep=2mm,
  column sep=3.5mm,
  lr delim={[ and ] around 2-2 to 5-5},
  lr delim={[ and ] around 2-6 to 5-6},
  ] (a) {  & (h,k) & (j,k) & (j,\ell) & (k,\ell) & (j,k,\ell) \\
h & \red{v}\color{red}\bm{(}\red{h,k} \color{red}\bm{)}  & w(j,k) & 0 & 0 & -w(j,k)  \\
j & w(h,k)  & \red{v}\color{red}\bm{(}\red{j,k} \color{red}\bm{)}  & \red{v}\color{red}\bm{(}\red{j,} \color{red}\bm{\ell)} & w(k,\ell) &  \red{v}\color{red}\bm{(}\red{h,k} \color{red}\bm{)}  \\
k & -v(k,h) & -v(k,j) & w(j,\ell)  & \red{v}\color{red}\bm{(}\red{k,}  \color{red}\bm{\ell)} & -v(\ell,k) \\
\ell & 0 & 0 & -v(\ell,j) & -v(\ell,k) & v(\ell,j) \\
};
\end{tikzpicture}
\end{center}
\caption{Illustration of the construction of the syzygy $\mathbf{s}_{(j,k,\ell)}^{(2)}$ in the proof of Theorem ~\ref{thm:SecondSyzygies}. Only the terms in bold and coloured in \red{red} are divisible by $X_1$, note that these terms can only come from $v(a,b)$ with $a<b$ and $(a,b)\neq (1,2)$. In the proof, it is shown that both $v(h,k)v(j,k)=\lt(g_k)$ and $v(j,\ell)v(\ell,k)=\lt(g_k)$; hence, $v(h,k)v(j,k)-v(j,\ell)v(\ell,k)=0$. Analogously both $v(h,k)v(k,j)=\lt(g_j)$ and $v(k,\ell)v(\ell,j)=\lt(g_j)$; so, $-v(h,k)v(k,j)+v(k,\ell)v(\ell,j)=0$. These two identities relate the middle rows of the matrix with columns $\mathbf{s}_{(h,k)}^{(1)},\ldots,\mathbf{s}_{(k,\ell)}^{(1)}$ to the column representing $\mathbf{s}_{(j,k,\ell)}^{(2)}$. Moreover, it is clearly to be seen that after a multiplication of these matrices, the top and bottom entries vanish. There remain only sums of products of entries that are not divisible by $X_1$ (which must reduce to zero as explained at the end of proof of  Theorem ~\ref{thm:SecondSyzygies}).} \label{fig:matrices}
\end{figure}


Since the second syzygy module of $\Rk(I)$ is free, the resolution stops here. The above results can be summarised in the main theorem of this section:

\begin{theorem}\label{thm:reso}
    There exists a free resolution of $\Rk(I)$ of the form
    $$ 0 \xrightarrow{\hspace{0.4cm}} S^{r-2} \xrightarrow{\hspace{0.2cm} \phi_2 \hspace{0.2cm} } S^{2(r-1)}\xrightarrow{\hspace{0.2cm} \phi_1 \hspace{0.2cm} } S^{r+1} \xrightarrow{\hspace{0.2cm} \phi_0 \hspace{0.2cm} } \Rk(I) \xrightarrow{\hspace{0.4cm}}  0$$
    encoded in the graph $\overline{\mathrm{Graph}}(\Rk(I))$, with differential maps $\phi_1$ and $\phi_2$ deduced in the usual way from Theorems \ref{thm:FirstSyzygies} and \ref{thm:SecondSyzygies}, respectively.
\end{theorem}

\subsection{The minimal free resolution of $\Rk(I)$}\label{sec:minimal_resolution}
As expected, the free resolution obtained in Theorem \ref{thm:reso} is not minimal, since we added a redundant element to the minimal generating set in order to obtain a Gr\"{o}bner basis. Nonetheless, it is very close to being minimal, in fact the minimal resolution can be recovered with a couple of simple algebraic reductions which do not depend on the parameters. These reductions are explicitly described in the present section.

Note that the nonzero constant coefficients in the differentials that keep this resolution from being minimal can only come from the coefficient $w(j,k)$ in $s^{(1)}_{(j,k)}$ and from $-w(j,k)$ in $s^{(2)}_{(j,k,l)}$. This is very straightforward since all other coefficients are of the type $v(a,b)$ and the l.c.m. of any two nonconstant terms can never give you a constant. Remember that $w(j,k)=\gcd(\mathrm{tt}(g_j),\mathrm{tt}(g_k))$, and note that $X_0,X_2 \mid \mathrm{tt}(g_\rho)$ for every $\rho \geq 3$. However, we know that: $\mathrm{tt}(g_0)=T_0^{d_1}X_2$, $\mathrm{tt}(g_1)=T_1^{u_2}X_0$, and $\mathrm{tt}(g_2)=T_0^{u_1}X_2$. Thus, $w(j,k)=\gcd(\mathrm{tt}(g_j),\mathrm{tt}(g_k))=1$ only for pairs $(0,1)$ and $(1,2)$. Nevertheless, there does not exist any first nor second syzygy where $j=0$ and $k=1$; hence, the only first syzygy with a nonzero constant coefficient is $s^{(1)}_{(1,2)}$ and the only second syzygy where that also happens is $s^{(2)}_{(1,2,3)}$. Since the binomials $g_0, g_1$ and $g_2$ are trivially obtained directly from $u_1,u_2,d_1$ and $d_2$, let us illustrate the matrices from the differentials as in Figure \ref{fig:matrices}:
\begin{figure}[ht!]\label{fig:reductions}
\begin{center}
\begin{tikzpicture}[
  delim yshift=1ex,
  delim xshift=.1em,]
\matrix[
  matrix of math nodes,
  row sep=2mm,
  column sep=2mm,
  lr delim={[ and ] around 2-2 to 6-6},
  lr delim={[ and ] around 2-8 to 6-9},
  ] (a)  {  & (0,2) & (1,2) & (1,3) & (2,3) & \dots & \hspace{0.6cm} \hspace{0.4cm}  & (1,2,3)& \dots \\
(0) & \hspace{0.2cm} \color{red}\bm{X_1} \hspace{0.2cm} & \textbf{1} & 0 & 0 &  0 \mydots 0 &\hspace{0.6cm}  (0,2) \hspace{-0.2cm} & \hspace{0.4cm}  \textbf{-1} \hspace{0.4cm}   & 0 \mydots 0   \\
(1) & T_0^{u_1}X_2  & T_1^{d_2-u_2} & \red{v}\color{red}\textbf{(1,3)} \ & w(2,3) & \dots  & \hspace{0.6cm} (1,2) \hspace{-0.2cm} & \red{v}\color{red}\textbf{(2,3)} & \dots  \\
(2) & -T_1^{u_2 X_0} & -T_0^{d_1-u_1} & w(1,3)  & \red{v}\color{red}\textbf{(2,3)} & \dots & \hspace{0.6cm} (1,3) \hspace{-0.2cm} & -v(3,2) & \dots  \\
(3) & 0 & 0 & -v(3,1) & -v(3,2) & \dots & \hspace{0.6cm} (2,3) \hspace{-0.2cm}& v(3,1) & \dots  \\
\vdots & \vdots  & \vdots & \vdots & \vdots & \hspace{0.2cm}  \ddots \hspace{0.2cm} & \hspace{0.6cm} \vdots \hspace{-0.2cm} &  \vdots & \hspace{0.2cm}  \ddots \hspace{0.2cm}   \\
};
\end{tikzpicture}
\end{center}
\end{figure}

On this figure one can clearly see that after two algebraic reductions one obtains the minimal free resolution. Moreover, if we take a closer look, the reduction on the first matrix gets rid of $s^{(1)}_{(1,2)}$ as well as our redundant generator $g_0$, i.e.\ erasing the whole first row on the first matrix with label $(0)$, and both the second column on the first matrix as well as the second row on the second matrix with labels $(1,2)$, while only altering the coefficients of one other first syzygy namely $s^{(1)}_{(0,2)}$ (the one represented on the first column). However, note that this same syzygy is removed immediately after due to the one reduction on the second matrix, together with the second syzygy $s^{(2)}_{(1,2,3)}$. Since, this first syzygy is not present in any of all other second syzygies, all entries in the matrix remain unchanged. Thus,  simply omitting these few columns and rows gives you the differential maps from the minimal free resolution of $\Rk(I)$.

\begin{theorem}\label{th:minimal_resolution}
    The minimal free resolution of $\Rk(I)$ of the form:
    $$ 0 \xrightarrow{\hspace{0.4cm}} S^{r-3} \xrightarrow{\hspace{0.2cm} \phi_2 \hspace{0.2cm} } S^{2(r-2)}\xrightarrow{\hspace{0.2cm} \phi_1 \hspace{0.2cm} } S^{r} \xrightarrow{\hspace{0.2cm} \phi_0 \hspace{0.2cm} } \Rk(I) \xrightarrow{\hspace{0.4cm}}  0$$
    encoded  in a graph of the form described in Section \ref{sec:graph}, with differential maps $\phi_1$ and $\phi_2$ deduced in the usual way from Theorems \ref{thm:FirstSyzygies} and \ref{thm:SecondSyzygies} respectively, acting on the reduced graph instead.
\end{theorem}
\begin{proof}
    The proof follows naturally from the explanations given above. The reductions on the free resolution from Theorem \ref{thm:reso} give us a result a minimal free resolution with one less element in the initial $S$-module of the sequence, two fewer elements in the generating set of the first module of the syzygies and one less in the second module of the sequences. This is equivalent to removing from $\overline{\mathrm{Graph}}(I)$ the node $g_0$, the edges $(0\rightarrow2)$ and $(1\rightarrow2)$, and the triangle $(1,2,3)$. Thus, we are left with the reduced graph $\mathrm{Graph}(\Rk(I))$ of the form described in Section \ref{sec:graph}, which supports the minimal free resolution in the same way since as we saw above the rest of the entries in the matrices remain unchanged.
\end{proof}

\section*{Acknowledgements}
This work is supported by grant PID2020- 116641GB-I00 funded by MCIN/ AEI/ 10.13039/501100011033.

\bibliographystyle{plain}
\bibliography{References.bib}

\begin{thebibliography}{10}

\bibitem{BerkeschSchreyer:SyzygiesFiniteLengthModulesRandomCurves}
Christine Berkesch and Frank-Olaf Schreyer.
\newblock Syzygies, finite length modules, and random curves.
\newblock In {\em Commutative algebra and noncommutative algebraic geometry. Volume {I}: Expository articles}, pages 25--52. Cambridge: Cambridge University Press, 2015.

\bibitem{BC17}
Winfried Bruns and Aldo Conca.
\newblock Linear resolutions of powers and products.
\newblock In Wolfram Decker, Gerhard Pfister, and Mathias Schulze, editors, {\em Singularities and Computer Algebra: Festschrift for Gert-Martin Greuel on the Occasion of his 70th Birthday}, pages 47--69. Springer International Publishing, Cham, 2017.

\bibitem{BCV15}
Winfried Bruns, Aldo Conca, and Matteo Varbaro.
\newblock Maximal minors and linear powers.
\newblock {\em Journal für die reine und angewandte Mathematik (Crelles Journal)}, 2015(702):41--53, 2015.

\bibitem{CKT07}
Hara Charalambous, Anargyros Katsabekis, and Apostolos Thoma.
\newblock Minimal systems of binomial generators and the indispensable complex of a toric ideal.
\newblock {\em Proceedings of the American Mathematical Society}, 135(11):3443--3451, 2007.

\bibitem{CortadellasDAndrea}
Teresa Cortadellas~Ben{\'{\i}}tez and Carlos D'Andrea.
\newblock The {Rees} algebra of a monomial plane parametrization.
\newblock {\em J. Symb. Comput.}, 70:71--105, 2015.

\bibitem{C08}
David~A. Cox.
\newblock The moving curve ideal and the {Rees} algebra.
\newblock {\em Theoretical Computer Science}, 392(1):23--36, 2008.

\bibitem{FL15}
Louiza Fouli and Kuei-Nuan Lin.
\newblock {Rees algebras of square-free monomial ideals}.
\newblock {\em Journal of Commutative Algebra}, 7(1):25 -- 53, 2015.

\bibitem{GV09}
Isidoro Gitler, Enrique Reyes, and Rafael~H. Villarreal.
\newblock Blowup algebras of square-free monomial ideals and some links to combinatorial optimization problems.
\newblock {\em Rocky Mountain Journal of Mathematics}, 39(1):71 -- 102, 2009.

\bibitem{IOSS24b}
Rodrigo Iglesias, Matthias Orth, Eduardo Sáenz-de Cabezón, and Werner~M. Seiler.
\newblock Homological invariants and involutive bases for the {R}ees algebra of some monomial ideals.
\newblock {\em In preparation}, 2024.

\bibitem{LaScalaStillman:StrategiesForComputingMinFreeRes}
Roberto {La Scala} and Michael {Stillman}.
\newblock Strategies for computing minimal free resolutions.
\newblock {\em {J. Symb. Comput.}}, 26(4):409--431, 1998.

\bibitem{LM06}
Ha~Minh Lam and Marcel Morales.
\newblock On the symmetric and {R}ees algebras of some binomial ideals.
\newblock {\em Vietnam Journal of Mathematics}, 31(1):63--70, 2006.

\bibitem{N21}
Lisa Nicklason.
\newblock On the {B}etti numbers and {R}ees algebras of ideals with linear powers.
\newblock {\em Journal of Algebraic Combinatorics}, 53:575--592, 2021.

\bibitem{OV10}
Ignacio Ojeda and Alberto Vigneron-Tenorio.
\newblock Indispensable binomials in semigroup ideals.
\newblock {\em Proceedings of the American Mathematical Society}, 138(12):4205--4216, 2010.

\bibitem{R01}
Tim R{\"o}mer.
\newblock {Homological properties of bigraded algebras}.
\newblock {\em Illinois Journal of Mathematics}, 45(4):1361 -- 1376, 2001.

\bibitem{SC95}
Thomas~W. Sederberg and Falai Chen.
\newblock Implicitization using moving curves and surfaces.
\newblock In {\em Proceedings of the 22nd Annual Conference on Computer Graphics and Interactive Techniques}, SIGGRAPH '95, page 301–308, New York, NY, USA, 1995. Association for Computing Machinery.

\bibitem{SGD97}
Thomas~W. Sederberg, Ron Goldman, and Hang Du.
\newblock Implicitizing rational curves by the method of moving algebraic curves.
\newblock {\em Journal of Symbolic Computation}, 23(2):153--175, 1997.

\bibitem{V95}
Rafael~H. Villarreal.
\newblock Rees algebras of edge ideals.
\newblock {\em Communications in Algebra}, 23(9):3513--3524, 1995.

\bibitem{V08}
Rafael~H. Villarreal.
\newblock Rees algebras and polyhedral cones of ideals of vertex covers of perfect graphs.
\newblock {\em J. Algebr. Comb.}, 27:293--305, 2008.

\end{thebibliography}

\end{document}